\documentclass[reqno]{amsart}

\usepackage{verbatim} 
\usepackage{amsmath}
\usepackage{amsfonts}
\usepackage{amssymb}
\usepackage{mathrsfs}
\usepackage{amsthm}
\usepackage{newlfont}
\usepackage{verbatim}

\usepackage{fullpage,amsmath, amssymb,amscd}
\usepackage{latexsym, epsfig, color}
\usepackage[mathscr]{eucal}
\usepackage{xypic, graphicx}
\usepackage{amscd}
\usepackage[all, cmtip]{xy}

\newcommand\cyr{%
 \renewcommand\rmdefault{wncyr}%
 \renewcommand\sfdefault{wncyss}%
 \renewcommand\encodingdefault{OT2}%
\normalfont\selectfont} \DeclareTextFontCommand{\textcyr}{\cyr}

\newtheorem{theorem}{Theorem}
\newtheorem{lemma}[theorem]{Lemma}
\newtheorem{corollary}[theorem]{Corollary}
\newtheorem{proposition}[theorem]{Proposition}

\theoremstyle{remark}
 \newtheorem{remark}[theorem]{Remark}

\def\Z{\mathbb Z}
\def\N{\mathbb N}
\def\Q{\mathbb Q}

\def\F{\mathbb F}

\def\cB{\mathcal B}
\def\cA{\mathcal A}
\def\cP{\mathcal P}

\def\fp{\mathfrak p}
\def\fc{\mathfrak c}
\def\fm{\mathfrak m}
\def\fP{\mathfrak P}

\def\fl{\mathfrak l}

\def\Ker{\operatorname{Ker}}

\def\deg{\operatorname{deg}}
\def\det{\operatorname{det}}
\def\tr{\operatorname{tr}}

\def\End{\operatorname{End}}

\def\Hom{\operatorname{Hom}}

\def\Gal{\operatorname{Gal}}

\def\Frob{\operatorname{Frob}}

\def\mod{\operatorname{mod}}
\def\dim{\operatorname{dim}}
\def\disc{\operatorname{disc}}

\def\exp{\operatorname{exp}}

\def\gcd{\operatorname{gcd}}

\def\GL{\operatorname{GL}}
\def\SL{\operatorname{SL}}
\def\PGL{\operatorname{PGL}}
\def\PSL{\operatorname{PSL}}

\def\O{\operatorname{O}}

\def\log{\operatorname{log}}

\def\ord{\operatorname{ord}}
\def\sep{\operatorname{sep}}

\def\ds{\displaystyle}

\begin{document}

\title{Drinfeld modules, Frobenius endomorphisms, and CM-liftings}


\author{
Alina Carmen Cojocaru and Mihran Papikian}
\address[Alina Carmen  Cojocaru]{
\begin{itemize}
\item[-]
Department of Mathematics, Statistics and Computer Science, University of Illinois at Chicago, 851 S Morgan St, 322
SEO, Chicago, 60607, IL, USA;
\item[-]
Institute of Mathematics  ``Simion Stoilow'' of the Romanian Academy, 21 Calea Grivitei St, Bucharest, 010702,
Sector 1, Romania
\end{itemize}
} \email[Alina Carmen  Cojocaru]{cojocaru@uic.edu}
\address[Mihran Papikian]{
\begin{itemize}
\item[-]
Department of Mathematics, 
Pennsylvania State University,
University Park, PA 16802, USA
\end{itemize}
} \email[Mihran Papikian]{papikian@math.psu.edu}

\thanks{
A.C. Cojocaru's  work on this material was partially supported by the National Science Foundation under agreements No. DMS-0747724 and No. DMS-0635607,
and by the European Research Council 
under Starting Grant 258713.}

\thanks{M. Papikian was supported in part by the Simons Foundation.}

\begin{abstract}
We give a global description of the Frobenius elements in the division fields of Drinfeld modules of rank $2$. 
We apply this description to derive a criterion for the splitting modulo primes of a class of non-solvable 
polynomials, and to study the frequency with which the reductions of 
Drinfeld modules have small endomorphism rings. We also generalize some of these results 
to higher rank Drinfeld modules and prove CM-lifting theorems for Drinfeld modules. 
\end{abstract}

\maketitle


\section{Introduction}

Given a finite Galois extension $L/K$ of global fields and a conjugacy class 
$C \subseteq \Gal(L/K)$, a fundamental problem is that of describing the (unramified) 
primes  $\mathfrak{p}$ of $K$ for which the conjugacy class  of the Frobenius  at $\mathfrak{p}$
is $C$. The Chebotarev 
Density Theorem provides the density $\#C/[L:K]$ of these primes, while, in general, the characterization 
of the primes themselves is a finer and deeper question.

One instance of a complete answer  to this question is that of the cyclotomics. For example, 
for $a$ an odd positive integer, $\Gal(\Q(\zeta_a)) \simeq (\Z/a \Z)^\times$, and  so  for any  rational prime $p \nmid a$, 
the Frobenius at $p$ is uniquely determined by the residue class of $p$ modulo $a$; 
in particular, $p$ splits completely in $\Q(\zeta_a)$ if and only if $p \equiv 1 (\mod a)$. A similar  result was proven by Hayes \cite{Ha} 
for the cyclotomic function fields introduced by Carlitz.

Natural extensions of the cyclotomics occur in the context of abelian varieties and Drinfeld modules through the division fields associated to these objects. For an abelian variety  of dimension 1 (an elliptic curve), defined over a global field, an explicit global characterization of the Frobenius in the division fields  of the variety has been  obtained using central results from the theory of complex  multiplication, and 
similarly to the case of the cyclotomics, there are numerous  applications  of this characterization
(cf.  \cite{Shimura} and \cite{DuTo}). 
For a higher dimensional abelian variety, the question  of describing explicitly the Frobenius  
in the division fields of the variety is open.
The focus of our paper is an investigation of this  question in \textit{the context of Drinfeld modules}, as described below.

Let $F$ be the function field of a smooth, projective, geometrically irreducible curve over 
the finite field $\F_q$ with $q$ elements. 
We distinguish a place $\infty$ of $F$, called the \textbf{place at infinity}, and we let  
$A$ denote the ring of functions in $F$ which have no poles away from $\infty$. 
Let $K$ be a field equipped with a homomorphism $\gamma: A\to K$. If $\gamma$ 
is injective, we say that $K$ has \textbf{$A$-characteristic $0$}; if $\ker(\gamma)=\fp\lhd A$ 
is a non-zero (prime) ideal, then we say that $K$ has \textbf{$A$-characteristic $\fp$}. 
Note that $K$ contains $\F_q$ as a subfield. Let $\tau$ be the Frobenius endomorphism of $K$ relative to $\F_q$, 
 i.e.,  the map $x\mapsto x^q$, 
 and let  $K\{\tau\}$ be  the non-commutative ring of polynomials 
 in the indeterminate $\tau$ with coefficients in $K$ and the  commutation rule  $\tau c=c^q \tau$ for any $c \in K$.
A {\bf{Drinfeld $A$-module over $K$}} is a ring homomorphism 
\begin{align*}
\psi: A &\to K\{\tau\} \\
a &\mapsto \psi_a=\gamma(a)+\sum_{1 \leq i \leq n_a} \alpha_i \tau^i, \quad \alpha_{n_a}\neq 0,
\end{align*}
whose image is not contained in $K$. One shows that there is an integer $r\geq 1$, called the \textbf{rank of $\psi$}, 
such that $n_a=r\log_q |a|_\infty$ for all $a\in A$, where $|\cdot|_\infty$ is the normalized valuation of $F$ defined by $\infty$; see \cite{Dr1}. 
Two Drinfeld modules, $\psi, \phi$, are \textbf{isomorphic} over $K$ if there exists $c\in K^\times$ such that $\psi_a=c^{-1}\phi_a c$ 
for all $a\in A$. 

Let $\psi$ be a Drinfeld module of rank $r$ over $F$, with  
$\gamma$ being the canonical embedding of $A$ into its fraction field $F$ (this shall be our setting throughout). 
We say $\psi$ has \textbf{good reduction} at the prime $\fp$ of $A$ if we can find $\phi$ over $F$ with the following properties:
\begin{itemize}
\item[(i)] $\phi$ is isomorphic to $\psi$ over $F$;
\item[(ii)] for all $a\in A$, the coefficients of $\phi_a$ are integral at $\fp$;
\item[(iii)] the map 
\begin{align*}
\phi \otimes \F_\fp: A &\to \F_\fp\{\tau\}\\
a &\mapsto \phi_a\mod \fp
\end{align*}
is a Drinfeld $A$-module of rank $r$ over $\F_\fp:=A/\fp$.  
\end{itemize}
Let $\cP_\psi$ denote the set of primes of good reduction of $\psi$. We will often implicitly assume that $\psi$ itself satisfies 
(ii) and (iii) at a given prime of good reduction. 

The \textbf{ring of $K$-endomorphisms} of $\psi$, $\End_K(\psi)$, is the centralizer  in $K\{\tau\}$ of the
image of $A$ under $\psi$. Denote by $F_\infty$ the completion of $F$ at $\infty$. The ring 
$\End_K(\psi)$ is a projective $A$-module of rank $\leq r^2$ with 
the property that $D := \End_K(\psi) \otimes_A F$ is a division algebra over $F$ such that $D\otimes_F F_\infty$ 
is also a division algebra (over $F_\infty$). Moreover, 
if $K$ has $A$-characteristic $0$, then $D$ is a field extension of $F$ of degree $\leq r$; see \cite{Dr1}. 
In this last case, the place $\infty$ does not split in the extension $D/F$. We call a finite field 
extension $F'$ of $F$ \textbf{imaginary} if $\infty$ does not split in $F'$. 

The Drinfeld module $\psi$ endows the algebraic closure $\overline{K}$ of $K$ with an $A$-module structure, where $a\in A$ 
acts by $\psi_a$. We shall write ${^\psi}\overline{K}$ if we wish 
to emphasize this action. The \textbf{$a$-torsion} $\psi[a]\subset \overline{K}$ of $\psi$ is the 
kernel of $\psi_a$, i.e., the set of zeros of the polynomial $\psi_a(x):=\gamma(a)x+\ds\sum_{1 \leq i \leq n_a} \alpha_i x^{q^i}\in K[x]$. 
The field $K(\psi[a])$, obtained by adjoining the elements of $\psi[a]$ to $K$, is called the \textbf{$a$-th division field} of $\psi$. 
 
 It is clear that $\psi[a]$ has a natural structure of an $A$-module. 
 Assume $a$ is coprime to $\ker(\gamma)$, if the $A$-characteristic of $K$ is non-zero.  
Then $\psi[a] \simeq_A  (A/aA)^{\oplus r}$ and $\psi[a]\subset K^\mathrm{sep}$ (since $\psi_a'(x)=\gamma(a)\neq 0$). 
The action of $G_K:=\Gal(K^\mathrm{sep}/K)$  
on $\psi[a]$ gives rise to a Galois representation 
\begin{equation}\label{eqResRep}
\bar{\rho}_{\psi,a}: G_K\to \GL_r(A/aA) . 
\end{equation}
In the theory of Drinfeld modules, the study of the division fields and the Galois representations associated to 
$\psi$ plays a central role. For example, when $r=1$, 
this study leads to explicit class field theory of $F$ (see \cite{Dr1}, \cite{Ha}). 

In this paper we mostly deal with Drinfeld modules for $A=\F_q[T]$, which, in some respects, is similar to that of elliptic curves over $\Q$. 
Our first goal is 
\textit{to provide an explicit global characterization of  the Frobenius at a prime $\mathfrak{p}$ of $F$ in the division fields of $\psi$} 
when $r=2$. We also give a less explicit version of this result which is valid for any $r\geq 2$. 
These results have several interesting applications, including a criterion for the splitting modulo primes of a class of non-solvable 
polynomials studied by Abhyankar. 
The second goal of the paper is 
\textit{to study the frequency with which the reductions of $\psi$ modulo $\mathfrak{p}$ have a small endomorphism ring}. 
This result opens up further important questions about the behaviour of the  reductions  of $\psi$ modulo primes and 
broadens a major theme of research related to the Sato-Tate conjecture and the 
Lang-Trotter conjectures. 
Finally, the third goal of the paper is \textit{to prove  CM-lifting theorems for general Drinfeld modules}, 
providing a function field counterpart of Deuring's Lifting Theorem. 

Now we give the precise statements of our main results. 

\begin{theorem}\label{global-artin}
Let $q$ be an odd prime power,  $A = \F_q[T]$ and $F = \F_q(T)$.
Let $\psi: A \to F\{\tau\}$ be a Drinfeld $A$-module over $F$, of rank 2.
Let $\mathfrak{p} = p A \in \cP_\psi$ be a prime of good reduction of $\psi$, where $p \in A$ is monic and irreducible. 
Let $a_{\mathfrak{p}}(\psi), b_{\mathfrak{p}}(\psi), \delta_{\mathfrak{p}}(\psi)$ be the 
following uniquely determined elements of $A$:
\begin{enumerate}
\item[(a)]
$a_{\mathfrak{p}}(\psi)$ is the coefficient of $x$  in the $\mathfrak{p}$-Weil polynomial of $\psi$,
$$
P_{\psi, \mathfrak{p}}(x) = x^2 + a_{\mathfrak{p}}(\psi) x + u_{\mathfrak{p}}(\psi) p \in A[x], 
$$
where $u_{\mathfrak{p}}(\psi) \in \F_q^\times$;
\item[(b)]
$b_{\mathfrak{p}}(\psi)$ is the unique monic polynomial such that, for any root 
$\pi_{\mathfrak{p}}(\psi)$ of $P_{\psi, \mathfrak{p}}$,
$$
\End_{\F_\fp} (\psi \otimes \F_\fp)/ A[\pi_{\mathfrak{p}}(\psi)]
\cong_A A/ b_\fp(\psi) A; 
$$
\item[(c)]
$\delta_{\mathfrak{p}}(\psi)$ is the unique generator of the discriminant ideal of
$\End_{  \F_{\mathfrak{p} }}  (\psi \otimes \F_{\mathfrak{p}})$
satisfying
$$
a_{\mathfrak{p}}(\psi)^2
-
4 u_{\mathfrak{p}}(\psi) p
=
b_{\mathfrak{p}}(\psi)^2 \delta_{\mathfrak{p}}(\psi).
$$
\end{enumerate} 
Then, for any $a \in A$ coprime to $p$, the reduction modulo $a$ of the matrix 
$$
\begin{pmatrix}
-\frac{a_{\mathfrak{p}}(\psi)}{2} & \frac{\delta_{\mathfrak{p}}(\psi) b_{\mathfrak{p}}(\psi)}{2}
\\
\frac{b_{\mathfrak{p}}(\psi)}{2} & -\frac{a_{\mathfrak{p}}(\psi)}{2}
\end{pmatrix}
\in M_2(A)
$$
represents the conjugacy class in $\GL_2(A/aA)$ of the image under $\bar{\rho}_{\psi, a}$ of the Frobenius at $\mathfrak{p}$ 
in the $a$-division field $F(\psi[a])$ of $\psi$. 
\end{theorem}

An immediate consequence to this result is a criterion for the splitting completely of a prime in 
$F(\psi[a])$, reminiscent of that for cyclotomic fields:
\begin{corollary}\label{split-compl}
In the setting of Theorem \ref{global-artin}, the prime  
$\mathfrak{p}$ splits completely in $F(\psi[a])/F$ if and only if 
$$
a_{\mathfrak{p}}(\psi) \equiv -2\ (\mod a)
$$
and
$$
b_{\mathfrak{p}}(\psi) \equiv 0\ (\mod a).
$$
\end{corollary}

Moreover,  we deduce the  $A$-module structure of $\F_\fp$ defined by the reduction $\psi \otimes \F_\fp$:
\begin{corollary}\label{structure-mod-p}
In the setting of Theorem \ref{global-artin}, the $A$-module structure ${^\psi}\F_\fp$ is given explicitly by
$$
{^\psi}\F_\fp \simeq_A A/d_{1, \mathfrak{p}} (\psi) A \times A/d_{2, \mathfrak{p}}(\psi) A,
$$
where
$$
d_{1, \mathfrak{p}}(\psi) 
=
\gcd
\left(
\frac{b_{\mathfrak{p}}(\psi)  }{2}, \frac{a_{\mathfrak{p}}(\psi)}{2} + 1
\right) \in A,
$$
$$
d_{2, \mathfrak{p}}(\psi)
=
\frac{1 + a_{\mathfrak{p}}(\psi) + u_{\mathfrak{p}}(\psi) p}{d_{1, {\mathfrak{p}}}(\psi)} \in A,
$$
and
$d_{1, \mathfrak{p}}(\psi) $ divides $d_{2, \mathfrak{p}}(\psi)$ (hence are uniquely determined up to a constant factor).
In particular, if $b_{\mathfrak{p}}(\psi) =1$, then ${^\psi}\F_\fp$ is $A$-cyclic.
\end{corollary}


\begin{theorem}\label{thmMain4}
Let $A = \F_q[T]$ and $F = \F_q(T)$. Let $\psi: A \to F\{\tau\}$ be a Drinfeld $A$-module over $F$, of rank $r\geq 2$.
Let $\mathfrak{p} = pA$ be a prime of good reduction of $\psi$, 
and $\pi_\fp(\psi)$ be any root of the $\fp$-Weil polynomial $P_{\psi, \fp}$ of $\psi$. 
\begin{itemize}
\item[(a)]
There are uniquely determined non-zero monic polynomials $b_{\fp, 1}(\psi), \dots, b_{\fp, r-1}(\psi)\in A$ such that
$$ 
\End_{\F_\fp}(\psi\otimes \F_\fp)/A[\pi_\fp(\psi)]\cong_A  A/b_{\fp, 1}(\psi)A\oplus\cdots\oplus A/b_{\fp, r-1}(\psi)A, 
$$
and 
$$
b_{\fp, i}(\psi) \text{ divides }b_{\fp, i+1}(\psi) \text{ for $i=1, \dots, r-2$}. 
$$
\item[(b)]
If $r$ is coprime to $q$, then 
$$
\disc(P_{\psi, \fp}) A=\disc(\End_{\F_\fp}(\psi\otimes \F_\fp))(b_{\fp, 1}(\psi)\cdots b_{\fp, r-1}(\psi))^2,
$$
where $\disc(P_{\psi, \fp})$ is the discriminant of the polynomial of $P_{\psi, \fp}$, and $\disc(\End_{\F_\fp}(\psi\otimes \F_\fp))$ 
is the discriminant ideal of $\End_{\F_\fp}(\psi\otimes \F_\fp)$. 
\item[(c)]
Assume $0\neq m\in A$ is coprime to $p$. Let $J_m$ be the subfield of $F(\psi[m])$ fixed by $\bar{\rho}_{\psi, m}(G_F)\cap Z(A/mA)$, 
where $Z(A/mA)$ denotes the center of $\GL_r(A/mA)$. Then $\fp$ splits completely in $J_m$ if and only if $m$ 
divides $b_{\fp,1}(\psi)$. 
\end{itemize}
\end{theorem}

Comparing (a) and (b) of Theorem \ref{thmMain4} with (b) and (c) of Theorem \ref{global-artin}, we 
see that the $r-1$ invariants $b_{\fp, 1}(\psi), \dots, b_{\fp, r-1}(\psi)$ generalize $b_\fp(\psi)$ to the rank-$r$ case. 
Although Theorem \ref{thmMain4} does not provide an explicit matrix for the Frobenius at $\fp$, part (c) 
of the theorem can be interpreted as a generalization of Corollary \ref{split-compl}. 
In the rank-$2$ case, $b_\fp(\psi)$ controls both $\End_{\F_\fp}(\psi\otimes\F_\fp)/A[\pi_\fp(\psi)]$  
and the splitting behavior of $\fp$ in division fields. 
In higher ranks, $b_{\fp, r-1}(\psi)$ controls the difference between the endomorphism rings, 
whereas $b_{\fp, 1}(\psi)$ controls the splitting of $\fp$. (Indeed, since all $b_{\fp, i}(\psi)$ divide $b_{\fp, r-1}(\psi)$, we have 
$\End_{\F_\fp}(\psi\otimes\F_\fp) \cong A[\pi_\fp(\psi)]$ if and only if $b_{\fp, r-1}(\psi)=1$.) 

An interesting arithmetic application of Theorems \ref{global-artin} and \ref{thmMain4} is a ``reciprocity law''  
for splitting of certain non-solvable polynomials in the style of Klein's  
approach to non-solvable quintics using elliptic curves (which itself 
is a generalization of a theorem of Gauss that the polynomial 
$x^3 - 2 \in \Z[x]$ splits completely modulo a rational prime $p \geq 5$ if and only 
if $p = \alpha^2 + 27 \beta^2$ for some integers $\alpha, \beta$). 
To introduce this class of polynomials, assume $A=\F_q[T]$ and let $\psi : A \to F\{\tau\}$ be a Drinfeld $A$-module of rank $r$ defined by
\begin{equation}\label{psiT}
\psi_T = T + g_1 \tau + g_2 \tau^2+\cdots + g_r \tau^r.
\end{equation}
Consider the polynomial
\begin{equation}\label{AbTri}
f_{\psi}(x) := T + g_1 x +  \cdots + g_r x^{(q^r-1)/(q-1)} \in F[x]
\end{equation}
obtained from $\psi_T(x)$ via the relation
$$
\psi_T(x) = x f_{\psi}\left(x^{q-1}\right).
$$

\begin{theorem}\label{abhyankar-thm} Assume $r\geq 2$ is coprime to $q$. 
\begin{enumerate}
\item[(a)] $f_{\psi}$ splits completely modulo 
$\fp\in \cP_\psi$ only if $T^2$ divides the discriminant of $P_{\psi, \fp}$. When $r=2$, this can be explicitly stated as 
$f_{\psi}$ splits completely modulo $\fp$ only if $p=u\alpha^2+T^2\beta$ for some $\alpha, \beta\in A$ and $u\in \F_q^\times$, 
where $p$ is the monic generator of $\fp$. 
\item[(b)]
Suppose $q\geq 5$, $r=2$, and $f_\psi(x)=gx^{q+1}+x+T$. If $g\in \F_q^\times$ 
or $g = h^{q-1}$ for some non-constant $h \in A$ not divisible by any prime of degree $1$ except possibly $T$, then 
the Galois group of $f_{\psi}$ over $F$ is isomorphic to $\PGL_2(\F_q)$,  and, in particular,  
is non-solvable.
\item[(c)] If $f_\psi(x)=x^{(q^r-1)/(q-1)}+\alpha Tx+T$, where $\alpha\in \F_q^\times$,  
then the Galois group of $f_{\psi}$ over $F$ is isomorphic to $\PGL_r(\F_q)$. 
\end{enumerate}
Under the assumptions in (b) or (c), the set of primes $\{\fp: \ b_{\fp, 1}(\psi)\equiv 0\ (\mod T)\}$
has Dirichlet density $\#\PGL_r(\F_q)^{-1}$.
\end{theorem}

Polynomials similar to $f_{\psi}(x)$ in (b) and (c) were extensively studied by Abhyankar 
in connection with the problem of resolution of singularities in positive characteristic (cf. \cite {Ab1}, \cite{Ab2});    
for that reason, we call them \textbf{Abhyankar trinomials}. In fact, the claim in part (c) of Theorem \ref{abhyankar-thm}
is a special case of Theorem 1.1 in \cite {Ab1}. The argument in \cite {Ab1} is somewhat hard to follow, 
mostly due to the generality Abhyankar aims for, but also because of frequent references to his other papers. 
For that reason, we give a proof of (b) by adapting Serre's methods 
for elliptic curves \cite{Serre} to Drinfeld modules. 


In the above results,  the invariants 
$a_{\mathfrak{p}}(\psi), b_{\mathfrak{p}}(\psi), \delta_{\mathfrak{p}}(\psi)$ associated to $\psi$
play an essential role.
The first one, ``the Frobenius trace'', has been the subject of several studies in relation to the 
Sato-Tate and Lang-Trotter Conjectures for Drinfeld modules 
(cf. \cite{Bro}, \cite{CoDa}, \cite{Da1}, \cite{Da2}, \cite{GekelerTAMS}, \cite{HsYu}, 
\cite{Po}, \cite{Yu2}, \cite{Zy}).  In this paper we study the second invariant, $b_{\mathfrak{p}}(\psi)$,
and prove:

\begin{theorem}\label{bp-one} Let the setting and notation be as in Theorem \ref{global-artin}. 
\begin{enumerate}
\item[(a)]
If $\End_{\overline{F}}(\psi) = A$, then, for $x \in \N$ going to infinity, we have the asymptotic formula
\begin{equation}\label{bp-asymptotic-nonCM}
\#\left\{
\mathfrak{p} \in {\mathcal{P}}_{\psi}:
\deg \mathfrak{p} = x,
\End_{\F_{\mathfrak{p}}}(\psi \otimes \F_{\mathfrak{p}})
=
A[\pi_{ \mathfrak{p}}(\psi)]
\right\}
\sim
\ds\sum_{m \in A \atop{m \; \text{monic}}}
\frac{\mu_A(m) c_{J_m}(x)}{[J_m : F]} 
\cdot
\frac{q^x}{x},
\end{equation}
where $\mu_A(\cdot)$ denotes the M\"{o}bius function on $A$, 
$J_m$ is the subfield of $F(\psi[m])$ fixed by the scalars, 
$c_{J_m} := [J_m \cap \overline{\F}_q : \F_q]$,
and
 \begin{equation*} 
  c_{J_m}(x)
  :=
  \left\{ \begin{array}{cc}
 c_{J_m} & \textrm{if } c_{J_m} | x, \\
  0  &    \textrm{otherwise}.     \\
\end{array}
 \right. \end{equation*}
 Moreover, the Dirichlet density of the set
$
\left\{
\mathfrak{p} \in {\mathcal{P}}_{\psi}:
\End_{\F_{\mathfrak{p}}}(\psi \otimes \F_{\mathfrak{p}})
=
A[\pi_{ \mathfrak{p}}(\psi)]
\right\}
$
exists and equals
$
\ds\sum_{m \in A \atop{m \; \text{monic} }  }
\frac{\mu_A(m)}{[J_m : F]}.
$
\item[(b)]
If $\End_{\overline{F}}(\psi)$ is the integral closure of $A$ in a quadratic imaginary 
extension $K$ of $F$, 
then, for $x \in \N$ going to infinity, we have the asymptotic formula
\begin{equation}\label{bp-asymptotic-CM}
\#\left\{
\mathfrak{p} \in {\mathcal{P}}_{\psi}:
\deg \mathfrak{p} = x,
\End_{\F_{\mathfrak{p}}}(\psi \otimes \F_{\mathfrak{p}})
=
A[\pi_{ \mathfrak{p}}(\psi)]
\right\}
\sim
\frac{c_K(x)}{2}
\cdot
\frac{q^x}{x},
\end{equation}
where 
$c_K := \left[K \cap \overline{\F}_q : \F_q\right]$
and
 \begin{equation*} 
  c_{K}(x)
  :=
  \left\{ \begin{array}{cc}
 c_{K} & \textrm{if } c_{K} | x, \\
  0  &    \textrm{otherwise}.     \\
\end{array}
 \right. \end{equation*}
 Moreover, the Dirichlet density of the set 
 $
\left\{
\mathfrak{p} \in {\mathcal{P}}_{\psi}:
\End_{\F_{\mathfrak{p}}}(\psi \otimes \F_{\mathfrak{p}})
=
A[\pi_{ \mathfrak{p}}(\psi)]
\right\}
$
exists and equals
$
\frac{1}{2}.
$
\end{enumerate}
\end{theorem}

Theorem \ref{global-artin} is the function field analogue of Theorem 2.1 in \cite{DuTo}. 
To prove this theorem, Duke and T\'oth use Deuring's Lifting Theorem. 
We avoid using such CM-liftings in the proof of Theorem \ref{global-artin} by 
exploiting the fact that a Drinfeld $A$-module of rank $r$ with endomorphism ring $A'$ 
can be considered as a Drinfeld $A'$-module of smaller rank. Nevertheless, 
the question of the existence of CM-liftings for Drinfeld modules is interesting. 
In this paper we prove the following analogue of Deuring's Lifting Theorem: 

\begin{theorem}\label{thmDLL}
Let $A$ be arbitrary, as at the beginning of this section. Let $k$ be a finite field with $A$-characteristic $\fp$. Let 
$\phi$ be a Drinfeld $A$-module of rank $2$ defined over $k$. Let $g\in \End_k(\phi)\backslash A$. 
Then there exist a discrete valuation field $K$ with $A$-characteristic $0$ and residue field $k$, a Drinfeld $A$-module 
$\psi$ of rank $2$ defined over $K$, and $f\in \End_K(\psi)$, such that $\phi$ 
with endomorphism $g$ is the reduction of $\psi$ with endomorphism $f$.  
\end{theorem}

In Section \ref{CMLifting}, we prove a general result about CM-liftings of Drinfeld modules of 
arbitrary rank from which Theorem \ref{thmDLL} follows. 
 The proofs of our main results are based on both algebraic and analytic techniques. 
 In particular, the proof of Theorem \ref{bp-one} is based on sieve methods such as the less standard Square Sieve, and, implicitly,  
 also on effective versions of the Chebotarev Density Theorem.

\section{Global description of the Frobenius: Proof of Theorem \ref{global-artin} and Theorem \ref{thmMain4}}\label{sProof1} 
\subsection{Preliminaries}\label{ss2.1}

Throughout this section, we assume that $A = \F_q[T]$. 
In addition to the notation in the introduction, we use the following: 
\begin{itemize}
\item
$A^{(1)}$ denotes the set of monic polynomials in $A$.
\item
For $0\neq a\in A$, let $\deg(a)$ be the degree of $a$ as a polynomial in $T$ and put $\deg(0):=-\infty$. 
\item
For $f = \frac{a}{b}\in F=\F_q(T)$, let $\deg(f) := \deg(a)-\deg(b)$. 
This  defines a valuation on $F$ with normalized norm $|f|_\infty:=q^{\deg(f)}$; 
the corresponding place of $F$ is $\infty$. 
\item
For a prime ideal $0\neq \fp  \lhd A$, let $F_\fp$ denote the completion of $F$ at $\fp$, 
$\F_{\mathfrak{p}}:=A/\fp$, 
$\deg(\fp):=[\F_\fp:\F_q]$.
\end{itemize}

Let $\psi : A \to F\{\tau\}$ be a Drinfeld $A$-module of rank $r$. 
Let $\fl = \ell A \lhd A$ be a prime ideal with generator $\ell \in A$. For an integer $n \geq 1$, we  define $\psi[\fl^n]:=\psi[\ell^n]$.   
(It is easy to see that this does not depend on the choice of $\ell$.) For $n'\geq n$ we have the inclusion
$\psi[\fl^n]\subseteq \psi[\fl^{n'}]$, which is compatible with the $A$-module structure and the action of $G_F$.
Hence 
$$
\psi[\fl^\infty] := \underset{\underset{n}{\rightarrow}}{\lim}\ \psi[\fl^n] \cong  (F_\fl/A_\fl)^{\oplus r},
$$
where $F_\fl$ and $A_\fl$ are the completions of $F$ and $A$ at $\fl$, respectively. 
The \textbf{$\fl$-adic Tate module of $\psi$}, defined as 
$$
T_\fl(\psi) := \Hom_{A_\fl}(F_\fl/A_\fl, \psi[\fl^\infty]) \cong A_\fl^{\oplus r},
$$ 
is endowed with a continuous action of $G_F$,  giving rise to a representation
 $$
 \rho_{\psi, \fl}: G_F \to \GL_2(A_\fl)
 $$ 
whose reduction modulo $\fl$ is $\bar{\rho}_{\psi, \ell}$ of (\ref{eqResRep}).  
 
Assume now that  $\fp\neq \fl$ is a prime of good reduction of $\psi$. 
More precisely, if $\psi$ is given by (\ref{psiT}), assume $\ord_{\fp}(g_i)\geq 0$ for all $1\leq i\leq r-1$ and $\ord_{\fp}(g_r)= 0$. 
Then, according to Theorem 1 in \cite{Ta}, the representation $\rho_{\psi, \fl}$ 
is unramified at $\fp$, and so, up to conjugation, there is a well-defined matrix 
$$
\rho_{\psi, \fl}(\Frob_\fp) \in \GL_r (A_{\fl})
$$
whose characteristic polynomial we denote by $P_{\psi, \mathfrak{p}}(x)$. 
The polynomial $P_{\psi, \fp}(x)$ has coefficients in $A$, does not depend on the choice of $\fl$, 
and is equal to the characteristic polynomial of the Frobenius endomorphism of the reduction 
$\psi\otimes \F_\fp$ acting on $T_\fl(\psi\otimes \F_\fp)$; see \cite[pp. 478-479]{Ta}. 
In particular, the roots of $P_{\psi, \fp}(x)$ are integral over $A$. Let $\pi_{\fp}(\psi)$ denote one of those roots.  

\begin{proposition}\label{propImR}
The field extension $F(\pi_{\fp}(\psi))/F$ is imaginary of degree $r$. 
\end{proposition} 
\begin{proof}
By the reduction properties of Drinfeld modules, $\pi:=\pi_{\fp}(\psi)$ is a Weil number of rank $r$ over $\F_\fp$; 
cf. \cite[p. 479]{Ta} and \cite[p. 165]{Yu}. 
Next, by the properties of Weil numbers, the extension $F(\pi)/F$ is imaginary of degree dividing $r$; see \cite[pp. 165-166]{Yu}.  
On the other hand, the norm $\mathrm{N}_{F(\pi)/F}(\pi)\in A$ generates the ideal $\fp^{[F(\pi):F]/r}$; cf. \cite[p. 167]{Yu}. 
Since $\fp$ is prime, we must have $[F(\pi):F]=r$. 
\end{proof}

Let $$E_{\psi, \mathfrak{p}} := \End_{\F_{\mathfrak{p}}}(\psi \otimes \F_{\mathfrak{p}}), 
\qquad \overline{E}_{\psi, \mathfrak{p}} := \End_{\overline{\F}_{\mathfrak{p}}}(\psi \otimes \F_{\mathfrak{p}}), 
$$ 
$$
{\mathcal{O}}_{\psi, \mathfrak{p}}:=\text{integral closure of $A$ in $F(\pi_{\fp}(\psi))$}. 
$$
As a consequence of \cite[Thm. 1]{Yu} and Proposition \ref{propImR}, we have $E_{\psi, \mathfrak{p}}\otimes_A F=F(\pi_{\fp}(\psi))$. 
Hence $A[\pi_{\fp}(\psi)]$ and $E_{\psi, \mathfrak{p}}$ are $A$-orders in $F(\pi_{\fp}(\psi))$, and we have the inclusions
\begin{equation}\label{rings}
A \subsetneq 
A[\pi_{\fp}(\psi)] \subseteq 
E_{\psi, \mathfrak{p}} \subseteq
{\mathcal{O}}_{\psi, \mathfrak{p}}. 
\end{equation}
It is known that $\overline{E}_{\psi, \mathfrak{p}}$ is a division algebra over $F$, and  
at the level of division algebras, we have the inclusions
\begin{equation}\label{fields}
F \subsetneq
F(\pi_{\fp}(\psi))
=
E_{\psi, \mathfrak{p}} \otimes_A F
=
{\mathcal{O}}_{\psi, \mathfrak{p}} \otimes_A F
\subseteq
\overline{E}_{\psi, \mathfrak{p}} \otimes_A F. 
\end{equation}
We say that $\mathfrak{p}$ is a {\bf{supersingular}} prime 
for $\psi$ if $\dim_F(\overline{E}_{\psi, \mathfrak{p}} \otimes_A F)=r^2$.  
We say that $\mathfrak{p}$ is an {\bf{ordinary}} prime 
for $\psi$ if $\dim_F(\overline{E}_{\psi, \mathfrak{p}} \otimes_A F)=r$.  If $r=2$, then any 
prime $\fp\in \cP_\psi$ is either ordinary or supersingular. 

When $r=2$, the coefficients of 
$$
P_{\psi, \mathfrak{p}}(x) = x^2 + a_{\fp}(\psi) x + a'_{\fp}(\psi) 
$$  
can be explicitly determined as follows.
Let $\mathrm{N}_{\F_\fp/\F_q}$ be the norm map from $\F_\fp$ to $\F_q$. Let 
$$
u_{\fp}(\psi) := (-1)^{\deg(\fp)} \mathrm{N}_{\F_\fp/\F_q}(g_2)^{-1},
$$ 
where, by abuse of notation, $g_2$ in the norm denotes the reduction of $g_2$ modulo $\fp$. 
For an integer $k\geq 1$, put $[k]:=T^{q^k}-T$, and define $s_0:=1$, $s_1:=g_1$,
$$
s_k:=-[k-1]s_{k-2}g_2^{q^{k-2}}+s_{k-1}g_1^{q^{k-1}}\quad (k\geq 2). 
$$ 

\begin{proposition}\label{FrobTrDet}
\noindent
\begin{enumerate}
\item[(i)]
The coefficient $a_{\fp}(\psi)\in A$ is  uniquely determined by
$$a_{\fp}(\psi) \equiv - u_{\fp}(\psi) s_{\deg(\fp)} (\mod \fp)$$ 
and 
\begin{equation}\label{RH}
\deg a_{\fp}(\psi) \leq \frac{ \deg(\fp)}{2}. 
\end{equation}
\item[(ii)]
The coefficient $a'_{\fp}(\psi) \in A$ is equal to $u_{\fp}(\psi) p$, where $p\in A^{(1)}$ is the monic generator of $\fp$. 
\end{enumerate}
\end{proposition}
\begin{proof}
This follows from Theorem 2.11 and Proposition 3.7 in \cite{GekelerTAMS}. 
\end{proof}

\subsection{Proof of Theorem \ref{global-artin} and its corollaries}
We keep the notation of $\S$\ref{ss2.1}, but assume that $r=2$ and $q$ is odd.  
 Note that even though the characteristic polynomial 
of $\bar{\rho}_{\psi,a}(\Frob_\fp)$ can be computed in terms of $g_1$, $g_2$ and $\fp$,
this is not sufficient for determining the conjugacy class of $\bar{\rho}_{\psi,a}(\Frob_\fp)$, 
as this matrix is not necessarily semi-simple. For this we need an extra 
invariant $b_{\fp}(\psi)$ related to the reduction of $\psi$ at $\fp$. 
Both $A[\pi_{\fp}(\psi)]$ and $E_{\psi, \fp}$ are 
$A$-orders in ${\mathcal{O}}_{\psi, \mathfrak{p}}$, hence  of the form
\begin{equation}\label{c}
A[\pi_{\fp}(\psi)]
=
A + \mathfrak{c}_{\fp}(\psi) {\mathcal{O}}_{\psi, \mathfrak{p}},
\end{equation}
\begin{equation}\label{c'}
E_{\psi, \fp}
=
A + \mathfrak{c}'_{\fp}(\psi) {\mathcal{O}}_{\psi, \mathfrak{p}}
\end{equation}
for some ideals $\mathfrak{c}_{\fp}(\psi)$, $\mathfrak{c}'_{\fp}(\psi)$ of $A$, satisfying
\begin{equation}\label{divisibility-c}
\mathfrak{c}'_{\fp}(\psi) \mid \mathfrak{c}_{\fp}(\psi).
\end{equation}
We  define
\begin{equation}\label{bp}
\mathfrak{b}_{\fp}(\psi) = b_{\fp}(\psi) A := 
\frac{\mathfrak{c}_{\fp}(\psi)}{\mathfrak{c}'_{\fp}(\psi)},
\end{equation}
where $b_{\fp}(\psi) \in A^{(1)}$.
This is an ideal of $A$  such that 
\begin{equation}\label{bp-measure}
E_{\psi, \fp}/A[\pi_{\fp}(\psi)]
\simeq
A/\mathfrak{b}_{\fp}(\psi).
\end{equation}
In other words, the ideal $\mathfrak{b}_{\fp}(\psi)$ measures how much larger the endomorphism ring 
$E_{\psi, \fp}$ is  than $A[\pi_{\fp}(\psi)]$.

\begin{proposition}\label{ap-bp}
Let $
\Delta(E_{\psi, \mathfrak{p}})$ denote the discriminant ideal of $E_{\psi, \fp}$.
Then, with  prior  notation, 
\begin{equation}\label{ap-bp-ideals}
\left(
a_{\fp}(\psi)^2
-
4 u_{\fp}(\psi) p
\right)
A
=
\mathfrak{b}_{\fp}(\psi)^2
\Delta(E_{\psi, \mathfrak{p}}).
\end{equation}
Consequently, there exists $\delta_{\fp}(\psi) \in A$ such that
\begin{equation*}
\Delta(E_{\psi, \mathfrak{p}}) = \delta_{\fp}(\psi) A
\end{equation*} 
and 
\begin{equation}\label{ap-bp-generators}
a_{\fp}(\psi)^2
-
4 u_{\fp}(\psi) p
=
b_{\fp}(\psi)^2 \delta_{\fp}(\psi).
\end{equation}
\end{proposition}
\begin{proof}
Let 
$\Delta(\mathcal{O}_{\psi, \mathfrak{p}})$
be the discriminant ideal of $\mathcal{O}_{\psi, \mathfrak{p}}$,
and let
$$d_\fp(\psi) := a_{\fp}(\psi)^2 - 4 u_{\fp}(\psi) p \in A$$
 be the discriminant of the characteristic polynomial $P_{\psi, \mathfrak{p}}$.
On one hand, by (\ref{c'}),
$$
\Delta(E_{\psi, \mathfrak{p}}) =
\mathfrak{c}'_{\fp}(\psi)^2 \Delta(\mathcal{O}_{\psi, \mathfrak{p}});
$$
hence, upon multiplying by $\mathfrak{c}_{\fp}(\psi)^2$ and using (\ref{bp}),
\begin{equation}\label{disc-relation1}
\mathfrak{b}_{\fp}(\psi)^2 \Delta(E_{\psi, \mathfrak{p}}) 
= 
\mathfrak{c}_{\fp}(\psi)^2 \Delta(\mathcal{O}_{\psi, \mathfrak{p}}).
\end{equation}
On the other hand, by (\ref{c}), 
\begin{equation}\label{disc-relation2}
d_{\fp}(\psi) A
=
\mathfrak{c}_{\fp}(\psi)^2 \Delta(\mathcal{O}_{\psi, \mathfrak{p}}).
\end{equation}
By putting (\ref{disc-relation1}) and  (\ref{disc-relation2}) together, we  complete the proof.
\end{proof}

\begin{proof}[Proof of Theorem \ref{global-artin}]
By definition,  $E_{\psi, \mathfrak{p}}$
is the centralizer of the image of $A$ under $\psi$ in $\F_{\mathfrak{p}}\{\tau\}$. Thus there exists a natural embedding
$$\phi: E_{\psi, \mathfrak{p}} \hookrightarrow \F_{\mathfrak{p}}\{\tau\}$$
such that the diagram 
$$
\xymatrix{
A \ar[r]  \ar[dr]_-{\psi \otimes \mathfrak{p}} &  E_{\psi, \mathfrak{p}}  \ar[d]^{\phi}  
\\
     & \F_{\mathfrak{p}}\{\tau\}
}
$$
is commutative.

Recalling  that $A \subsetneq E_{\psi, \mathfrak{p}}$ and using that  $E_{\psi, \mathfrak{p}}$ is an 
$A$-module of rank 2, while $\psi$ is a Drinfeld $A$-module of rank 2,   we see that $\phi$ defines an 
elliptic $E_{\psi, \mathfrak{p}}$-module over $\F_{\mathfrak{p}}$ of rank 1 in the sense of Definition 2.1 in \cite{Ha}.
We will use $\phi$ to determine the action of the Frobenius 
of $\Gal(\overline{\F}_{\mathfrak{p}}/\F_{\mathfrak{p}})$ on $\psi[a]$.

On one hand, since $(a, p) = 1$,
  we have an isomorphism of $E_{\psi, \mathfrak{p}}$-modules
$
\phi[a] \simeq_{E_{\psi, \mathfrak{p}}} E_{\psi, \mathfrak{p}}/a E_{\psi, \mathfrak{p}}.
$
On the other hand, from the commutative diagram, we have
$
\psi[a] \simeq_{E_{\psi, \mathfrak{p}}} \phi[a].
$
Thus
$$
\psi[a] \simeq_{E_{\psi, \mathfrak{p}}} E_{\psi, \mathfrak{p}}/a E_{\psi, \mathfrak{p}}.
$$
Under this isomorphism, the action of the Frobenius of $\Gal(\overline{\F}_{\mathfrak{p}}/\F_{\mathfrak{p}})$ on 
$\psi[a]$ corresponds to multiplication by $\pi_{\fp}(\psi)$
on $E_{\mathfrak{p}} (\psi)/a E_{\psi, \mathfrak{p}}$.

We now explore how this action extends to the $A$-module structure of $\psi[a]$.
We fix a square root $\sqrt{\delta_{\fp}(\psi)}$ of $\delta_{\fp}(\psi)$ in $F^{\text{sep}}$ and write
$$
E_{\psi, \mathfrak{p}} = A + \sqrt{\delta_{\fp}(\psi)} A.
$$
By  (\ref{ap-bp-generators}), 
\begin{equation}\label{pi-1}
\pi_{\fp}(\psi) = -\frac{a_{\fp}(\psi)}{2} + \sqrt{\delta_{\fp}(\psi)} \frac{b_{\fp}(\psi)}{2}
\in
E_{\psi, \mathfrak{p}},
\end{equation}
and so the action of $\pi_{\fp}(\psi)$ on the $A$-module $E_{\psi, \mathfrak{p}}$ is given by
(\ref{pi-1})  and
$$
\pi_{\fp}(\psi) \sqrt{\delta_{\fp}(\psi)} = 
\frac{\delta_{\fp}(\psi) b_{\fp}(\psi)}{2}
 +
 \sqrt{\delta_{\fp}(\psi)} \left(-\frac{a_{\fp}(\psi)}{2}\right).
$$
This completes the proof of Theorem \ref{global-artin}. 
\end{proof}

Corollary \ref{split-compl} is an immediate consequence to Theorem \ref{global-artin}.
The description of $d_{1, \fp}(\psi)$  in
Corollary \ref{structure-mod-p} is a consequence to Corollary \ref{split-compl} and  
the property that, for $a \in A$ with $(a, p) = 1$,
$\mathfrak{p}$ splits completely in $F(\psi[a])/F$ if and only if 
$A/a A \times A /a A$ is isomorphic to an $A$-submodule of ${^\psi}\F_{\mathfrak{p}}$; 
see \cite[Prop. 23]{CoSh}.
The description of $d_{2, \fp}(\psi)$ in Corollary \ref{structure-mod-p} is a consequence
of 
$$
P_{\psi, \mathfrak{p}}(1) A = \chi({^\psi}\F_{\mathfrak{p}})=d_{1, \fp}(\psi)d_{2, \fp}(\psi) A,
$$
 where $\chi({^\psi}\F_{\mathfrak{p}})$ denotes the Euler-Poincar\'{e} characteristic of ${^\psi}\F_{\mathfrak{p}}$ (see \cite{Ge}).
 
\subsection{Proof of Theorem \ref{thmMain4}}

To simplify the notation, let $\pi:=\pi_\fp(\psi)$ and $E:=E_{\psi, \fp}$. 
From Proposition \ref{propImR} and (\ref{rings}) we get that $A[\pi]\subseteq E$ are $A$-orders in $F(\pi)$. 
Since $A$ is a principal ideal domain and $[F(\pi):F]=r$, the $A$-modules $A[\pi]$ and $E$ are free of rank $r$.  
Now by the elementary divisors theorem \cite[Theorem III.7.8]{Lang}, there is an exact sequence of $A$-modules 
\begin{equation}\label{eq-HRankb}
0\longrightarrow A[\pi]\longrightarrow E \longrightarrow 
 A/b_0 A\oplus A/b_1 A\oplus\cdots\oplus A/b_{r-1}A\longrightarrow 0,
\end{equation}
where $b_0, \dots, b_{r-1}\in A$ are uniquely determined monic polynomials such that 
\begin{equation}\label{eq-div}
b_0 \mid b_{1}\mid \cdots \mid b_{r-1}. 
\end{equation}
(Of course, the $b_i$'s depend on $\psi$ and $\fp$, which we omit from notation.) 
Note that every element of $A[\pi]$, considered as an element of $E$, 
is a multiple of $b_0$. But $1\in A[\pi]$, so $b_0=1$. In other terms, $A/b_0A$ is trivial, and can be ignored. This proves part (a) of the theorem. 

If we assume that $r$ is coprime to $q$, then the extension $F(\pi)/F$ is separable.  
The elementary properties of discriminants then imply (cf. \cite[Exercise VI.32]{Lang}) 
$$
\disc(P_{\psi, \fp}) A=\disc(A[\pi])=\disc(E)(b_{\fp, 1}(\psi)\cdots b_{\fp, r-1}(\psi))^2,
$$
which is (b). 

As in the rank-$2$ case, we have an isomorphism $\psi[m] \simeq_{E} E/mE$ with the action of $\bar{\rho}_{\psi, m}(\Frob_\fp)$ 
on the left hand side corresponding to multiplication by $\pi$ on the right hand side. Consider the $A$-linear 
transformation of the free rank-$r$ $A$-module $E$ induced by multiplication by $\pi$. This transformation 
is congruent to an element of the center $Z(A)\cong A$ of $M_r(A)$ modulo $m$ if and only if $A[\pi]\subseteq A+mE$. 
On the other hand, $A[\pi]\subseteq A+mE$ if and only if $(E/mE)/(A[\pi]/(A[\pi]\cap mE))\cong (A/mA)^{\oplus r-1}$. 
Tensoring (\ref{eq-HRankb}) with $A/mA$, we see that this last condition is equivalent to 
$$
(A/b_1A\otimes_A A/mA)\oplus\cdots\oplus (A/b_{r-1}A\otimes_A A/mA) \cong (A/mA)^{\oplus r-1}. 
$$
As is easy to check, 
$$
A/b_iA\otimes_A A/mA\cong A/\gcd(b_i, m)A. 
$$
Thus, $\Frob_\fp$ acts trivially on $J_m$ (equivalently,  $\fp$ splits completely in $J_m$) 
if and only if $m$ divides all $b_i$. Since $b_1$ divides all $b_i$, this 
last condition is equivalent to $m|b_1$. This concludes the proof of the theorem. 

\section{Abhyankar trimonials: Proof of Theorem \ref{abhyankar-thm}}

Let $\psi$ and $f_\psi$ be as in (\ref{psiT}) and (\ref{AbTri}), respectively. 
Let $\Gal(f_{\psi})$ denote the Galois group of the splitting field of $f_{\psi}$ over $F$.
We consider the composition
of $\bar{\rho}_{\psi, T}$ with the natural projection onto $\PGL_r(A/TA)$, and, after identifying 
$A/T A \simeq \F_q$,  we write it  as
$$
\hat{\rho}_{\psi, T}: G_F \longrightarrow \PGL_r(\F_q).
$$
If $0\neq s\in \psi[T]$, then $s^{q-1}$ is a zero of $f_\psi(x)$. 
The center $Z(\F_q) \simeq \F_q^\times$ of $\GL_2(\F_q)$ acts on $\psi[T]$ by the usual
multiplication, i.e., $\gamma\in \F_q^\times$ maps $s$ to $\gamma s\in F^\mathrm{sep}$. Hence 
$\gamma$ maps $s^{q-1}$ to $\gamma^{q-1}s^{q-1}=s^{q-1}$, so the action of $\GL_r(\F_q)$ 
on the set of zeros of $f_\psi$, induced from the action on $\psi[T]$, factors through $\PGL_r(\F_q)$.  
This implies that the action of $G_F$ on the set of zeros of $f_\psi$ factors through $\hat{\rho}_{\psi, T}$, and 
\begin{equation}\label{galoispsi}
\Gal(f_{\psi}) \simeq \hat{\rho}_{\psi, T}(G_F).
\end{equation} 

Now let $\fp \in {\cal{P}}_{\psi}$, $\fp \neq T$. 
It follows from (\ref{galoispsi}) and Theorem \ref{thmMain4} (c) that 
$f_{\psi}$ splits completely modulo $\fp$ if and only if $b_{\fp,1}(\psi) \equiv 0\ (\mod T)$.
Therefore, part (a) of Theorem \ref{abhyankar-thm} follows from Theorem \ref{thmMain4} (b) and (\ref{ap-bp-generators}).

Now we focus on part (b) of Theorem \ref{abhyankar-thm}. Let $\psi$ be the rank-$2$ Drinfeld module defined by 
$\psi_T = T + \tau + g \tau^2$. Our goal is to prove that, provided either $g \in \F_q^\times$ 
or $g = h^{q-1}$ for some non-constant $h \in A$ not divisible by any prime 
of degree $1$ except possibly $T$,
\begin{equation}\label{image}
\hat{\rho}_{\psi, T}(G_F) \simeq \PGL_2(\F_q).
\end{equation}
For this, we will follow the general strategy of \cite[$\S$2.8]{Serre}.

Let us consider the case $g \in \F_q^\times$. Then $\psi$ has good reduction at every prime
of $A$, and so the extension $F(\psi[T])/F$ is unramified away from $T$ and $\infty$.
In particular, it is unramified  at every prime $\fp = p A$ defined by $p = T - c$ for some 
$c \in \F_q^\times$.  For such $\fp$ let us outline a few properties of $\bar{\rho}_{\psi, T} (\Frob_{\fp})$, which will
eventually restrict the possible group  structures of  $\hat{\rho}_{\psi, T}(G_F)$.
By Proposition \ref{FrobTrDet}, 
$$
\det \bar{\rho}_{\psi, T} (\Frob_{\fp})
= 
u_{\fp}(\psi) p \mod T
=
\frac{c}{g}. 
$$
Therefore
\begin{equation}\label{eqDetmap}
\det \bar{\rho}_{\psi, T} :
G_F \longrightarrow \F_q^\times \quad \text{is surjective.}
\end{equation}
Again, by Proposition \ref{FrobTrDet}, 
\begin{equation}\label{eqTrrhobar}
\tr \bar{\rho}_{\psi, T} (\Frob_{\fp})
=
- a_{\fp}(\psi)=-\frac{1}{g} \in \F_q^\times.
\end{equation}
Hence
\begin{equation}\label{first-restr}
d_\fp(\psi)= a_{\fp}(\psi)^2 - 4 u_{\fp}(\psi) p
=
\frac{1}{g^2} - \frac{4c}{ g},
\end{equation} 
\begin{equation}\label{second-restr}
t_\fp(\psi)
:= \frac{\tr \bar{\rho}_{\psi, T} (\Frob_{\fp})^2}{\det \bar{\rho}_{\psi, T} (\Frob_{\fp})}=
\frac{a_{\fp}(\psi)^2}{u_{\fp}(\psi) p}
=
\frac{1}{c g}.
\end{equation}
Since $q$ is odd, (\ref{first-restr}) implies that $d_\fp(\psi)$ assumes all values of 
$\F_q \backslash \left\{\frac{1}{g^2}\right\}$. In particular, since $q \geq 5$, 
\begin{equation}\label{sq}
\text{there are $\fp$ for which $d_\fp(\psi)$ is a non-zero square}
\end{equation}
and
\begin{equation}\label{nonsq}
\text{there are $\fp$ for which $d_\fp(\psi)$ is not a square}. 
\end{equation}
Moreover, (\ref{second-restr}) implies that there are $\fp$ for which   
\begin{equation}\label{fvalues}
t_\fp(\psi)\not\in \{0, 1, 2, 4\}
\end{equation} 
and
\begin{equation}\label{fequation}
t_\fp(\psi)\; \text{does not satisfy $u^2 - 3 u + 1 = 0$}
\end{equation}
(for example, if the characteristic 
is not $3$, then $c := (3g)^{-1}$ gives the value $3$, which satisfies these restrictions). 

We will use the following classical theorem:

\begin{theorem}[Dickson]
Any proper subgroup of $\PGL_2(\F_q)$ is contained in one of the groups:
\begin{itemize}
\item[(i)] a Borel subgroup;
\item[(ii)] $\PSL_2(\F_q)$; 
\item[(iii)] a conjugate of the subgroup $\mathrm{PGL}_2(\F)$ for some subfield $\F\subsetneq \F_q$;
\item[(iv)] a dihedral group $D_{2n}$ of order $2n$, where $n$ is not divisible by the characteristic of $\F_q$;
\item[(v)] a subgroup isomorphic to one of the permutation groups $A_4$, $A_5$, $S_4$. 
\end{itemize}
\end{theorem}
\begin{proof}
See \cite[Thm. 8.27]{Huppert}. 
\end{proof}

The properties of $H := \hat{\rho}_{\psi, T} (G_F)$ derived from the above observations will exclude all cases 
in Dickson's theorem, leaving $H=\PGL_2(\F_q)$ as the only possibility.   
Indeed, (i) is not possible by (\ref{nonsq}), and (ii) is not possible by (\ref{eqDetmap}). 
If $H$ is conjugate to a subgroup of $\mathrm{PGL}_2(\F)$, then $t_\fp(\psi)\in \F$ for all $p=T-c$.  
This contradicts the fact that $t_\fp(\psi)=(cg)^{-1}$ assumes all values in $\F_q^\times$ as $c$ varies. Hence (iii) 
is not possible. If $H$ is isomorphic to $A_4$, $A_5$ or $S_4$, then for each $h\in H$, the element $u=\tr(h)^2/\det(h)$ 
is equal to $0,1,2,4$ or satisfies $u^2-3u+1=0$; this follows from \cite[$\S$2.6]{Serre}, although in \cite{Serre} 
this is stated for prime fields. Hence (v) is not possible by (\ref{fvalues}) and (\ref{fequation}). 
Finally, to exclude (iv) we argue as in \cite[p. 284]{Serre}. If $H$ is cyclic or dihedral, then $\bar{\rho}_{\psi, T}(G_F)$ 
is contained in a normalizer of a Cartan subgroup of $\GL_2(\F_q)$. But the trace of $\bar{\rho}_{\psi, T}(\Frob_\fp)$ 
is non-zero by (\ref{eqTrrhobar}), and by (\ref{nonsq}) there is $\fp$ for which $d_\fp(\psi)$ is not a square; this leads to  
a contradiction as in \cite{Serre}.  

To prove that $\hat{\rho}_{\psi, T} (G_F)=\PGL_2(\F_q)$ when 
$g = h^{q-1}$ for some non-constant $h \in A$ not divisible by any prime 
of degree $1$, except possibly $T$, one can use 
the same arguments as above, based on the calculations (below, $\fp=T-c$ with $c\in \F_q^\times$)
$$
\det(\bar{\rho}_{\psi, T}(\Frob_{\fp})) = (-1)h(c)^{-(q-1)}(T-c) = c \in \F_q^\times,
$$
$$
\tr(\bar{\rho}_{\psi, T}(\Frob_{\fp})) = -\frac{1}{h(c)^{q-1}} = -1 \in \F_q^\times. 
$$

As we mentioned in the introduction, part (c) of Theorem \ref{abhyankar-thm} follows from the main 
result in \cite{Ab1}. Finally, for a Drinfeld module $\psi$ producing $f_\psi$ in part (b) or (c), the Chebotarev Density Theorem implies 
that the Dirichlet density of 
$$
\{\fp \in {\cal{P}}_{\psi}: b_{\fp,1}(\psi) \equiv 0\ (\mod T) 
\}= \{\fp \in {\cal{P}}_{\psi}: f_{\psi} \; \text{splits completely modulo $\fp$} \}
$$
exists and equals $\frac{1}{\#\PGL_r(\F_q)}$.

\section{Reductions of Drinfeld modules: Proof of Theorem \ref{bp-one}}

\subsection{Preliminaries}

The proofs of the following two lemmas are elementary and are left to the reader:

\begin{lemma}\label{lemma1}
Let $y \geq 1$ be an integer. Then:
\begin{enumerate}
\item[(i)]
$
\ds\sum_{m \in A^{(1)} \atop{0 \leq \deg m \leq y}
}
1
=
\frac{q^{y+1} - 1}{q-1};
$
\item[(ii)]
$
\ds\sum_{m \in A^{(1)} \atop{0 \leq \deg m \leq y}
}
\deg m
\leq
y \frac{q^{y+1} - 1}{q-1}.
$
\end{enumerate}
\end{lemma}

\begin{lemma}\label{lemma2}
Let $y \geq 3$ be an integer and let $\alpha > 1$.
Then:
\begin{enumerate}
\item[(i)]
$
\ds\sum_{a \in A \atop{\deg a > y}}
\frac{1}{q^{\alpha \deg a}}
=
\frac{q}{\left(1 - \frac{1}{q^{\alpha - 1}}\right) q^{(\alpha - 1)(y + 1)}};
$
\item[(ii)]
$
\ds\sum_{a \in A \atop{\deg a > y}}
\frac{\log \deg a}{q^{\alpha \deg a}}
\leq
\frac{\log y}{(\alpha - 1) q^{(\alpha - 1) y} \log q }
+
\frac{1}{y (\alpha - 1)^2 q^{(\alpha - 1) y} (\log q)^2}, 
$ provided that $(\alpha-1)y\log q\log y>~1$. 
\end{enumerate}
\end{lemma}

\begin{lemma}\label{lemma-tau}
Let $h \in A \backslash \F_q$ and $\tau_A(h) := \ds\sum_{d \in A \atop{d | h}} 1$ its divisor function. Then, for any $\varepsilon > 0$,
$$
\tau_A(h) \ll_{\varepsilon} |h|_{\infty}^{\varepsilon}.
$$
\end{lemma}
\begin{proof}
Over $\Z$,  this is a well-known result (see, for example, the proof in \cite[p. 344]{HaWr}).
Over $A$, one can prove the result in essentially the same way. We include the details for completeness.

Consider the prime factorization $h = u \ds\prod_{\ell | h} \ell^{\alpha}$ of $h$. Then
$$
\frac{\tau_A(h)}{|h|_{\infty}^{\varepsilon}}
=
\ds\prod_{\ell | h}
\frac{\alpha + 1}{|\ell|_{\infty}^{\alpha \varepsilon}}
=
\ds\prod_{\ell | h 
\atop{
|\ell|_{\infty} < 2^{\frac{1}{\varepsilon}}
}}
\frac{\alpha + 1}{|\ell|_{\infty}^{\alpha \varepsilon}}
\cdot
\ds\prod_{\ell | h 
\atop{
|\ell|_{\infty}  \geq 2^{\frac{1}{\varepsilon}}
}}
\frac{\alpha + 1}{|\ell|_{\infty}^{\alpha \varepsilon}}
\leq
\ds\prod_{\ell | h 
\atop{
|\ell|_{\infty} < 2^{\frac{1}{\varepsilon}}
}}
\frac{\alpha + 1}{|\ell|_{\infty}^{\alpha \varepsilon}}
\cdot
\ds\prod_{\ell | h 
\atop{
|\ell|_{\infty}  \geq 2^{\frac{1}{\varepsilon}}
}}
\frac{\alpha + 1}{2^{\alpha}}
\leq
\ds\prod_{\ell | h 
\atop{
|\ell|_{\infty} < 2^{\frac{1}{\varepsilon}}
}}
\frac{\alpha + 1}{|\ell|_{\infty}^{\alpha \varepsilon}}.
$$

Observe that
$$
\alpha \varepsilon \log 2 
\leq
\exp(\alpha \varepsilon \log 2)
=
2^{\alpha \varepsilon}
\leq
|\ell|_{\infty}^{\alpha \varepsilon},
$$
therefore
$$
\frac{
\alpha
}{
|\ell|_{\infty}^{\alpha \varepsilon}      
}
\leq
\frac{1}{\varepsilon \log 2}.
$$
For $\ell$ satisfying $|\ell|_{\infty} < 2^{\frac{1}{\varepsilon}}$, we thus obtain
$$
\frac{\alpha + 1}{|\ell|_{\infty}^{\alpha \varepsilon}}
\leq
\frac{1}{\varepsilon \log 2}
+
1
\leq
\exp\left(\frac{1}{\varepsilon \log 2}\right).
$$
This gives
$$
\ds\prod_{\ell | h 
\atop{
|\ell|_{\infty} < 2^{\frac{1}{\varepsilon}}
}}
\frac{\alpha + 1}{|\ell|_{\infty}^{\alpha \varepsilon}}
\leq
\exp\left(
\frac{1}{\varepsilon \log 2}
\cdot
\#\{\ell | h : |\ell|_{\infty} < 2^{\frac{1}{\varepsilon}}\}
\right)
\leq
\exp\left(
\frac{1}{\varepsilon \log 2}
\cdot
\frac{q^{2^{\frac{1}{\varepsilon}}} - 1}{q-1}
\right),
 $$
 a constant in $q$ and $\varepsilon$.
\end{proof}

Now let  us fix $a \in A \backslash \F_q$ and, as before, consider $F_a :=  F(\psi[a])$ and $J_a \subseteq F_a$ introduced in
part (c) of Theorem \ref{thmMain4}.
This field may also be  understood  by  considering the composition of $\bar{\rho}_{\psi, a}$ with the projection onto 
$\PGL_2(A/a A)$. 
Indeed, this composition leads to a Galois representation
$$
\hat{\rho}_{\psi, a} :
G_F \longrightarrow \PGL_2(A/a A)
$$
satisfying
$$
J_a = \left(F^{\text{sep}}\right)^{\Ker \hat{\rho}_{\psi, a}}.
$$
(Note that we have already considered the special case $\hat{\rho}_{\psi, T}$ in the proof of 
Theorem \ref{abhyankar-thm}.)

In what follows, we  recall some more properties of the extensions  $F_a/F$ and $J_a/F$:

\begin{theorem}\label{properties-division-fields}

\noindent
\begin{enumerate}
\item[(i)]
The degrees of the fields of constants of $F_a$ and $J_a$, that is,
$$
c_{F_a} := [F_a \cap \overline{\F}_q : \F_q],
$$
\begin{equation}\label{def-cJd}
c_{J_a} := [J_a \cap \overline{\F}_q : \F_q],
\end{equation}
are uniformly bounded from above in terms of $\psi$. That is,
$$
c_{J_a} \leq c_{F_a} \leq C(\psi)
$$
for some constant $C(\psi) \in \N \backslash \{0\}$.
\item[(ii)]
The genera $g_{F_a}$, $g_{J_a}$ of $F_a$, $J_a$ are bounded from above by
$$
g_{J_a} \leq g_{F_a} \leq  G(\psi) \; \# \GL_2(A/a A)\;  \deg a
$$ 
for some constant $G(\psi) \in \N \backslash \{0\}$.
\item[(iii)]
The degrees of $F_a/F$, $J_a/F$ are bounded from above by
$$
[F_a : F] \leq \#\GL_2(A/a A),
$$
$$
[J_a : F] \leq \#\PGL_2(A/a A).
$$
\item[(iv)]
Assume that $\End_{\overline{F}}(\psi) = A$. 
There exists $M(\psi) \in A^{(1)}$ such that, if   $(a, M(\psi)) = 1$, then
$$
\Gal(F_a/F) \simeq \GL_2(A/a A),
$$
$$
\Gal(J_a/F) \simeq \PGL_2 (A/ a A)
$$
and 
$$
c_{F_a} = c_{J_a} = 1;
$$
if $a$ arbitrary, then
$$
\frac{|a|_{\infty}^4}{\log \deg a + \log \log q}
\ll_{\psi}
[F_a : F]
\leq
|a|_{\infty}^4,
$$
$$
\frac{|a|_{\infty}^3}{\log \deg a + \log \log q}
\ll_{\psi}
[J_a : F]
\leq
|a|_{\infty}^3.
$$
\item[(v)]
Assume that $\End_{\overline{F}}(\psi) \neq A$. 
Then 
$$
\frac{|a|_{\infty}^2}{\log \deg a + \log \log q}
\ll_{\psi}
[F_a : F]
\ll_{\psi}
|a|_{\infty}^2,
$$
$$
\frac{|a|_{\infty}}{\log \deg a + \log \log q}
\ll_{\psi}
[J_a : F]
\ll_{\psi}
|a|_{\infty}.
$$
\item [(vi)]
For $x \in \N$, let
$$
\Pi_{1}(x, F_a/F):=
\#\{
\mathfrak{p} \in {\mathcal{P}}_{\psi} : p \nmid a, 
\deg \mathfrak{p} = x, 
\mathfrak{p} \; \text{splits completely in $F_a$}
\},
$$
$$
\Pi_{1}(x, J_a/F):=
\#\{
\mathfrak{p} \in {\mathcal{P}}_{\psi} : p \nmid a, 
\deg \mathfrak{p} = x,
\mathfrak{p} \; \text{splits completely in $J_a$}
\}.
$$
Then
$$
\Pi_{1}(x, F_a/F)
=
\frac{c_{F_a}(x)}{[F_a : F]} \cdot \frac{q^x}{x}
+
\O_{\psi}\left( \frac{q^{\frac{x}{2}}}{x} \deg a\right),
$$
$$
\Pi_{\hat{C}}(x, J_a/F)
=
\frac{c_{J_a}(x)}{[J_a : F]} \cdot \frac{q^x}{x}
+
\O_{\psi}\left( \frac{q^{\frac{x}{2}}}{x} \deg a\right),
$$
where
 $$
  c_{F_a}(x)
  :=
  \left\{ \begin{array}{cl}
 c_{F_a} & \textrm{if } c_{F_a} | x, \\
  0  &    \textrm{else},     \\
\end{array}
 \right. 
 $$
  $$
  c_{J_a}(x)
  :=
  \left\{ \begin{array}{cl}
 c_{J_a} & \textrm{if } c_{J_a} | x, \\
  0  &    \textrm{else}.     \\
\end{array}
 \right. 
$$

\item[(vii)]
Let $\bar{C}$, $\hat{C}$ be   conjugacy classes in $\Gal(F_a/F)$, respectively in $\Gal(J_a/F)$.
Denote by $a_{\bar{C}}$, $a_{\hat{C}}$ respectively, 
a positive integer such that, for any $\sigma \in \Gal(F_a/F)$, $\Gal(J_a/F)$ respectively,
  the restriction of $\sigma$ to 
$F_a \cap \overline{\F}_q$, $J_a \cap \overline{\F}_q$ respectively,
equals the corresponding restriction of $\tau^{a_{\bar{C}}}$,   $\tau^{a_{\hat{C}}}$ respectively.
For $x \in \N$, let
$$
\Pi_{\bar{C}}(x, F_a/F):=
\#\{
\mathfrak{p} \in {\mathcal{P}}_{\psi} : p \nmid a, 
\deg \mathfrak{p} = x, 
\sigma_{\mathfrak{p}} \subseteq \bar{C}
\},
$$
$$
\Pi_{\hat{C}}(x, J_a/F):=
\#\{
\mathfrak{p} \in {\mathcal{P}}_{\psi} : p \nmid a, 
\deg \mathfrak{p} = x,
\sigma_{\mathfrak{p}} \subseteq \hat{C}
\}.
$$
Then
$$
\Pi_{\bar{C}}(x, F_a/F)
=
\frac{c_{F_a}(x) \cdot  \#\bar{C}}{[F_a : F]} \cdot \frac{q^x}{x}
+
\O_{\psi}\left((\# \bar{C})^{\frac{1}{2}}  q^{\frac{x}{2}} \deg a\right),
$$
$$
\Pi_{\hat{C}}(x, J_a/F)
=
\frac{c_{J_a}(x)  \cdot \#\hat{C}}{[J_a : F]} \cdot \frac{q^x}{x}
+
\O_{\psi}\left((\# \hat{C})^{\frac{1}{2}}  q^{\frac{x}{2}} \deg a\right),
$$
where
 \begin{equation}\label{def-cFd-x} 
  c_{F_a}(x)
  :=
  \left\{ \begin{array}{cl}
 c_{F_a} & \textrm{if } c_{F_a} | x - a_{\bar{C}}, \\
  0  &    \textrm{else},     \\
\end{array}
 \right. \end{equation}
  \begin{equation} \label{def-cJd-x}
  c_{J_a}(x)
  :=
  \left\{ \begin{array}{cl}
 c_{J_a} & \textrm{if } c_{J_a} | x - a_{\hat{C}}, \\
  0  &    \textrm{else}.     \\
\end{array}
 \right. \end{equation}
 Note that this notation generalizes the one in part (vi).
 Moreover, note that, part (vii) holds also for unions of conjugacy classes.
 \end{enumerate}
\end{theorem}
\begin{proof}
For part (i), see \cite[Remark 7.1.9]{Go}. 
For part (ii), see \cite[Cor. 7]{Ga}.
Part (iii) follows from the injectivity of the residual representations $\Gal (F_a/F) \longrightarrow \GL_2(A/aA)$
and $\Gal (J_a/F) \longrightarrow \PGL_2(A/aA)$. 
The claims about $\Gal (F_a/F)$, $\Gal (J_a/F)$, $[F_a:F]$, and $[J_a:F]$ in parts (iv) and (v) can 
be derived from the main results of \cite{PiRu}, as explained in 
\cite[Section 3.6]{CoSh}. The fact that $c_{F_a} = c_{J_a} = 1$ then follows from Proposition \ref{propGeomExt}. 
Parts (vi) and 
(vii) are applications of the effective Chebotarev Density Theorem of \cite{MuSc}, 
 as well as of the prior parts of Theorem \ref{properties-division-fields}; see \cite[Section 4]{CoDa},
  \cite[Section 4]{CoSh} for more details.
 That part (vii) holds also for unions of conjugacy classes  can be seen by 
modifying the proof in \cite{MuSc} by using the techniques 
of \cite[Section 3]{MuMuSa}.
\end{proof}

\begin{proposition}\label{propGeomExt}
Let $A=\F_q[T]$ and $F=\F_q(T)$. Let $\psi: A\to F\{\tau\}$ be a Drinfeld module of rank $r$ defined over $F$. 
Assume $\psi$ has good reduction at the primes dividing $a\in A$ and 
$\Gal(F_a/F)\cong \GL_r(A/aA)$. In addition, if $r=2$, assume $q$ is odd. 
Under these assumptions, the extension $F_a$ of $F$ is geometric, i.e., $\F_q$ is algebraically closed in $F_a$.  
\end{proposition}
\begin{proof}
Let $aA=\prod_{i=1}^m \fp_i^{s_i}$ be the prime decomposition of the ideal $aA$. Since there is an isomorphism of groups
$$
\GL_r(A/aA)\cong \prod_{i=1}^m \GL_r(A/\fp_i^{s_i}), 
$$
the commutator of $\GL_r(A/aA)$ is the direct product of the commutators of $\GL_r(A/\fp_i^{s_i})$. On the other hand, 
since the set of nonunits in $A/\fp_i^{s_i}$ forms an ideal, 
according to \cite{Litoff} we have 
$$
[\GL_r(A/\fp_i^{s_i}), \GL_r(A/\fp_i^{s_i})]=\SL_r(A/\fp_i^{s_i}). 
$$
(Here we implicitly use the assumption that if $r=2$, then $q$ is odd.) This implies that 
\begin{equation}\label{eqpropGeom1}
[\GL_r(A/aA), \GL_r(A/aA)]=\SL_r(A/aA).
\end{equation}
We also have the exact sequence
\begin{equation}\label{eqpropGeom2}
0\to \SL_r(A/aA)\to \GL_r(A/aA)\xrightarrow{\det} (A/aA)^\times\to 0. 
\end{equation}

By assumption, $\Gal(F_a/F)\cong \GL_r(A/aA)$. Let $K$ be the subfield of $F_a$ fixed by $\SL_r(A/aA)$. 
Let $\F$ be the algebraic closure of $\F_q$ in $F_a$, and let $F'=\F F$. The extension $F'/F$ 
is Galois with Galois group isomorphic to $\Gal(\F/\F_q)$; in particular, it is cyclic. Due to (\ref{eqpropGeom1}), 
the field $F'$ must be a subfield of $K$, as $\Gal(F'/F)$ is a quotient group of $\Gal(F_a/F)$ which is abelian. 
Thus, it is enough to show that $K/F$ is geometric. 

There exists a Drinfeld $A$-module $\phi$ of rank-$1$ defined over $F$ such that 
there is an isomorphism of $\Gal(F^\mathrm{sep}/F)$-modules (cf. \cite{He})
$$
\phi[a]\cong \bigwedge^r \psi[a].
$$
Thus, $F(\phi[a])$ is the subfield of $F_a$ fixed by the kernel of the determinant on $\GL_r(A/aA)$. Therefore, 
due to (\ref{eqpropGeom2}), $K=F(\phi[a])$ and $\Gal(F(\phi[a])/F)\cong (A/aA)^\times$. 
Since $\psi$ has good reduction at the primes dividing $a$, the same is true for $\phi$. Finally, by Proposition 5.2 in \cite{HayesTAMS}, 
$F(\phi[a])/F$ is geometric. 
\end{proof}

\begin{remark}
In general, a composition of geometric extensions need not be geometric, so in the previous proof we cannot 
immediately reduce to the case when $aA=\fp^s$. 
\end{remark}

\subsection{Proof of Part (a) of Theorem \ref{bp-one}
}

Let
\begin{equation}\label{defB}
{\mathcal{B}}(\psi, x)
:=
\#\left\{
\mathfrak{p} \in {\mathcal{P}}_{\psi} :
\deg \mathfrak{p} = x,
E_{\psi, \mathfrak{p}}
=
A[\pi_{\fp}(\psi)]
\right\}.
\end{equation}
Our goal is to derive an explicit asymptotic formula for ${\mathcal{B}}(\psi, x)$, when $q$ is fixed and
$x \rightarrow \infty$.
We start with the simple remarks that
\begin{eqnarray*}
{\mathcal{B}}(\psi, x)
&=&
\#\left\{
\mathfrak{p} \in {\mathcal{P}}_{\psi} :
\deg \mathfrak{p} = x,
b_{\fp}(\psi) = 1
\right\}
\\
&=&
\#\left\{
\mathfrak{p} \in {\mathcal{P}}_{\psi} :
\deg \mathfrak{p} = x,
\ell \nmid b_{\fp}(\psi) \; \forall \ell \in A^{(1)}
\right\}
\\
&=&
\ds\sum_{m \in A^{(1)}}
\mu_A(m)
\#\left\{
\mathfrak{p} \in {\mathcal{P}}_{\psi} :
\deg \mathfrak{p} = x,
m \mid b_{\fp}(\psi) 
\right\},
\end{eqnarray*} 
where  in the first line we used (\ref{bp-measure}).

An essential aspect in the asymptotic  study of such sums  is that of determining the range of the 
polynomial $m \in A^{(1)}$ under summation as a function of $x$. By combining the property
$m \mid b_{\fp}(\psi)$
with (\ref{ap-bp-generators}),
we obtain
$$
m^2 \mid a_{\fp}(\psi)^2 - 4 u_{\fp}(\psi) p.
$$
Upon recalling (\ref{RH})
and using that $\deg \mathfrak{p} = x$, we deduce that
$\deg m \leq \frac{x}{2}$.
Thus
$$
{\mathcal{B}}(\psi, x)
=
\ds\sum_{
m \in A^{(1)}
\atop{\deg m \leq \frac{x}{2}}
}
\mu_A(m)
\#\left\{
\mathfrak{p} \in {\mathcal{P}}_{\psi} :
\deg \mathfrak{p} = x,
m \mid b_{\fp}(\psi) 
\right\}.
$$

By Theorem \ref{global-artin},   the extension $J_m/F$ 
has the property that, for 
any  $\mathfrak{p} = p A \in {\mathcal{P}}_{\psi}$ with $(p, m) = 1$,
\begin{equation}\label{eqb_psplit}
m \mid b_{\fp}(\psi)
\text{ if and only if $\fp$ splits completely in $J_m$}.
\end{equation}
(Note that, if $\deg \mathfrak{p} = x$ and $\deg m \leq \frac{x}{2}$, then the generator $p$ of $\mathfrak{p}$ is coprime
 with $m$; hence $\mathfrak{p}$ is unramified in $J_m$.)
Consequently, we can write
\begin{eqnarray}\label{B-split}
{\mathcal{B}}(\psi, x)
&=&
\ds\sum_{
m \in A^{(1)}
\atop{\deg m \leq y}
}
\mu_A(m) \Pi_1(x, J_m/F)
+
\ds\sum_{
m \in A^{(1)}
\atop{
y < \deg m \leq \frac{x}{2}
}
}
\mu_A(m) 
\#\left\{\mathfrak{p} \in {\mathcal{P}}_{\psi} :
\deg \mathfrak{p} = x, m \mid b_{\fp}(\psi)
\right\},
\end{eqnarray}
where
$y = y(x)$ is a parameter to be chosen optimally later as a function of $q$ and $x$, and
$$
\Pi_{1}(x, J_m/F) :=
\#
\left\{
\mathfrak{p} \in {\mathcal{P}}_{\psi} :
(p, m) = 1,
\deg \mathfrak{p} = x,
\mathfrak{p} \; \text{splits completely in $J_m/F$}
\right\}.
$$

The  splitting of ${\mathcal{B}}(\psi, x)$ in two sums is guided by the natural  strategy of using an effective version of the Chebotarev Density Theorem, and by the limitation of this tool for  our problem. In particular, the Chebotarev Density Theorem can be used only for estimating the first sum on the right-hand side of ${\mathcal{B}}(\psi, x)$ above, while other methods must be developed to estimate the remaining sum. These latter methods constitute the heart  of the proof.

\subsubsection{The main term of ${\mathcal{B}}(\psi, x)$}

For $y = y(x) \leq \frac{x}{2}$ a parameter, we  focus on
$$
{\mathcal{B}}_1(\psi, x, y) :=
\ds\sum_{
m \in A^{(1)}
\atop{
\deg m \leq y
}
}
\mu_A(m) \Pi_1(x, J_m/F).
$$ 

By part (vi) of Theorem \ref{properties-division-fields}, this becomes
\begin{eqnarray*}
{\mathcal{B}}_1(\psi, x, y)
&=&
\ds\sum_{
m \in A^{(1)}
\atop{
\deg m \leq y
}
}
\frac{\mu_A(m) c_{J_m}(x)}{[J_m : F]}
\cdot
\frac{q^x}{x}
+
\O_{\psi}
\left(
\ds\sum_{
m \in A^{(1)} \; \text{squarefree}
\atop{
\deg m \leq y
}
}
\frac{q^{\frac{x}{2}}}{x} \deg m
\right)
\\
&=&
\ds\sum_{m \in A^{(1)}}
\frac{\mu_A(m) c_{J_m}(x)}{[J_m : F]}
\cdot
\frac{q^x}{x}
-
\ds\sum_{
m \in A^{(1)}
\atop{
\deg m > y
}
}
\frac{\mu_A(m) c_{J_m}(x)}{[J_m : F]}
\cdot
\frac{q^x}{x}
+
\O_{\psi}\left(q^{\frac{x}{2} + y}\right),
\end{eqnarray*}
where, in the last line we used part (ii) of Lemma \ref{lemma1}.

To estimate the middle term, we use parts (i) and (iv) of Theorem \ref{properties-division-fields}, as well as Lemma \ref{lemma2},
 and obtain 
$$
\ds\sum_{
m \in A^{(1)}
\atop{
\deg m > y
}
}
\frac{\mu_A(m) c_{J_m}(x)}{[J_m : F]}
\ll_{\psi}
\ds\sum_{
m \in A^{(1)} \; \text{squarefree}
\atop{
\deg m > y
}
}
\frac{\log \deg m + \log \log q}{q^{3 \deg m}}
\ll
\frac{\log y}{q^{2 y} \log q}.
$$

In summary,
\begin{equation}\label{B1}
{\mathcal{B}}_1(\psi, x, y)
=
\ds\sum_{m \in A^{(1)}}
\frac{\mu_A(m) c_{J_m}(x)}{[J_m : F]} \cdot \frac{q^x}{x} 
+
\O_{\psi}\left(q^{\frac{x}{2} + y}\right)
+
\O_{\psi}\left(q^{x - 2y}\right).
\end{equation}

\subsubsection{The error term of ${\mathcal{B}}(\psi, x)$}

For $y = y(x) \leq \frac{x}{2}$,  we   focus on obtaining an upper bound for 
\begin{equation*}
{\mathcal{B}}_{2}(\psi, x, y) 
:=
\ds\sum_{
m \in A^{(1)}
\atop{y < \deg m \leq \frac{x}{2}}
}
\mu_A(m) 
\#\left\{\mathfrak{p} \in {\mathcal{P}}_{\psi} : \deg \mathfrak{p} = x, m \mid b_{\fp}(\psi)\right\}.
\end{equation*}
By (\ref{ap-bp-generators}),
$$
m \mid b_{\fp}(\psi)
\; \; \Rightarrow \;  \; 
m^2 \mid \left(a_{\fp}(\psi)^2 - 4 u_{\fp}(\psi) p\right).
$$
Thus there exist $f, g \in A$ with $g$ squarefree such that
$$
a_{\fp}(\psi)^2 - 4 u_{\fp}(\psi) p = m^2 f^2 g.
$$
Upon relabeling 
$h := m f$,
we obtain that
$$
{\mathcal{B}}_2(\psi, x, y) 
\leq
\ds\sum_{
h \in A
\atop{
y < \deg h \leq \frac{x}{2}
}
}
\tau_A(h)
\#\left\{
\mathfrak{p} \in {\mathcal{P}}_{\psi}:
\deg \mathfrak{p} = x,
\;
\text{$\exists g \in A$ squarefree  such that $a_{\fp}(\psi)^2 - 4 u_{\fp}(\psi) p = h^2 g$}
\right\}.
$$

The above range for $\deg h$ is determined simply from 
$$
\deg h = \deg m + \deg f,
$$
hence from
$$
\deg h \geq \deg m > y,
$$
and also from
$$
h^2 \mid \left(a_{\fp}(\psi)^2 - 4 u_{\fp}(\psi) p\right),
$$
hence from
$$
2 \deg h \leq \deg p = x,
$$
after recalling (\ref{RH}).

Using Lemma \ref{lemma-tau}, we deduce  that
\begin{equation*}
{\mathcal{B}}_2(\psi, x, y)
\ll_{\varepsilon}
q^{\varepsilon x}
\ds\sum_{
h \in A
\atop{y < \deg h \leq \frac{x}{2}}
}
\#\left\{
\mathfrak{p} \in {\mathcal{P}}_{\psi} :
\deg \mathfrak{p} = x,
\;
\text{$\exists g \in A$ squarefree such that $a_{\fp}(\psi)^2 - 4 u_{\fp}(\psi) p = h^2 g$}
\right\}.
\end{equation*}

Note that the factorization
$
a_{\fp}(\psi)^2 - 4 u_{\fp}(\psi) p = h^2 g
$
is unique up to the multiplication of $g$ by a square in $\F_q^{\times}$.
As such, 
\begin{eqnarray*}
{\mathcal{B}}_2(\psi, x, y)
&\ll_{\varepsilon}&
q^{x \varepsilon}
\ds\sum_{
g \in A \; \text{squarefree}
\atop{\deg g < x - 2 y}
}
\#\left\{
\mathfrak{p} \in {\mathcal{P}}_{\psi} :
\deg \mathfrak{p} = x,
g \left(a_{\fp}(\psi)^2 - 4 u_{\fp}(\psi) p\right) \;
\text{is a square in $A$}
\right\}
\\
&=:&
q^{x \varepsilon}
\ds\sum_{
g \in A \; \text{squarefree}
\atop{
\deg g < x - 2 y
}
}
S_g(\psi).
\end{eqnarray*}
The range of $\deg g$ above is obtained once again using (\ref{RH}):
$$
2 \deg h + \deg g \leq x \; \Rightarrow \; \deg g \leq x - 2 \deg h < x - 2 y.
$$

To estimate $S_g(\psi)$ we  rely on the function field  analogue of the Square Sieve  proven in 
\cite[Section 7]{CoDa} and on part (vii) of Theorem \ref{properties-division-fields}. Specifically, we use the resulting bound
\begin{equation}\label{square-sieve-application}
S_{g}(\psi)
\ll
q^{\frac{7 x}{8}}
(x + \deg g)
+
q^{\frac{3 x}{4}} x \left(x + \deg g\right)^2
\end{equation}
(which we will prove shortly)
and  deduce that
\begin{eqnarray}\label{B2}
{\mathcal{B}}_{2}(\psi, x, y)
&\ll_{\psi, \varepsilon}&
q^{
\frac{15 x}{8}
- 2 y 
+
x \varepsilon
}
x^3. 
\end{eqnarray}

Now let us prove (\ref{square-sieve-application}); our arguments use tools from  \cite[Sections 7,  8]{CoDa} 
and are included in detail for completeness. We recall the Square Sieve for $A$:

\begin{theorem}\label{square-sieve}
Let ${\mathcal{A}}$ be a finite multiset of  non-zero elements of $A$.
Let ${\mathcal{P}}$ be a finite set of  primes of $A$. Let
$$
S({\mathcal{A}}) := \{a \in {\mathcal{A}}: a = b^2 \; \text{for some $b \in A$} \},
$$
and for any $a \in A$ define
$$
\nu_{{\mathcal{P}}}(a) := \#\{\ell \in {\mathcal{P}}:  \ell \mid a\}.
$$
Then
\begin{eqnarray*}
\#S({\mathcal{A}})
&\leq&
 \frac{\#{\mathcal{A}}}{\#{\mathcal{P}}} 
 + 
 \max_{
\ell_1, \ell_2 \in {\mathcal{P}}
 \atop 
 \ell_1 \neq \ell_2 } 
 \left| 
 \ds\sum_{a \in {\mathcal{A}}}
\left(\frac{a}{\ell_1}\right)
\left(\frac{a}{\ell_2}\right) 
\right|
\\
&+& 
\frac{2}{\#{\mathcal{P}}}
 \ds\sum_{a \in {\mathcal{A}}}
\nu_{{\mathcal{P}}}(a) 
+ 
\frac{1}{(\#{\mathcal{P}})^2} \ds\sum_{a \in{\mathcal{A}}} \nu_{{\mathcal{P}}}(a)^2.
\end{eqnarray*}
\end{theorem}

 We apply Theorem \ref{square-sieve} in the setting
$$
{\mathcal{A}} := 
\left\{ 
g \left( a_{\fp}(\psi)^2 -  4 u_{\fp}(\psi) p\right) : 
\mathfrak{p} \in  {\mathcal{P}}_{\psi}, \; \deg \mathfrak{p} = x 
\right\}
$$
and
$$
{\mathcal{P}} := \left\{ {\ell} \in A: {\ell} \; \text{prime},
\; \deg {\ell} = \theta \right\}
$$
for some parameter $\theta = \theta(x) <  x$, to be chosen
optimally later.

We obtain
\begin{eqnarray}\label{square-sieve-applied}
S_g(\psi)
&\leq& 
\frac{\#{\mathcal{A}}}{\#{\mathcal{P}}} 
+
 \max_{
{\ell}_1, {\ell}_2 \in {\mathcal{P}}
 \atop {\ell}_1 \neq {\ell}_2 
} 
\left|
 {\ds\sum_{ \mathfrak{p} \in {\mathcal{P}}_{\psi}
\atop \deg \mathfrak{p} = x }
}
\left( 
\frac{ 
g
\left(
a_{\fp}(\psi)^2 - 4 u_{\fp}(\psi) p 
\right)
}{
{\ell}_1 }
\right)
\left( 
\frac{ 
g
\left(
a_{\fp}(\psi)^2 - 4 u_{\fp}(\psi) p 
\right)
}{
{\ell}_2 }
\right)
\right|
\nonumber
\\
&+& 
\frac{
2
}{
\#{\mathcal{P}}
} 
\ds\sum_{
\mathfrak{p} \in {\mathcal{P}}_{\psi}
\atop{\deg \mathfrak{p} = x}
}
\nu_{{\mathcal{P}}}
\left(
g
\left(
a_{\fp}(\psi)^2 - 4 u_{\fp}(\psi) p 
\right)
\right) 
+ 
\frac{1
}{
(\#{\mathcal{P}})^2
} 
\ds\sum_{
\mathfrak{p} \in {\mathcal{P}}_{\psi}
\atop{\deg \mathfrak{p} = x}
}
\nu_{{\mathcal{P}}}
\left(
g
\left(
a_{\fp}(\psi)^2 - 4 u_{\fp}(\psi) p 
\right)
\right)^2. 
\end{eqnarray}

On one hand, by the Prime Number Theorem for $A$,
\begin{equation}\label{sqsieve-1}
\frac{\#{\mathcal{A}}}{\#{\mathcal{P}}} \asymp q^{x - \theta}
\frac{\theta}{x}.
\end{equation}

On the other hand, by noting that, for  any $a \in A$,
$\nu_{{\mathcal{P}}}(a) \leq \deg a$,
and by using  (\ref{RH}), we deduce that, for primes $\mathfrak{p}$ of degree $x$, 
\begin{equation*}
\nu_{{\mathcal{P}}}
\left(
g
\left(
a_{\fp}(\psi)^2 - 4 u_{\fp}(\psi) p 
\right)
\right)
 \leq 
 x + \deg g.
\end{equation*}
We infer the estimates
\begin{equation}\label{sqsieve-2}
\frac{
2
}{
\#{\mathcal{P}}
} 
\ds\sum_{
\mathfrak{p} \in {\mathcal{P}}_{\psi}
\atop{\deg \mathfrak{p} = x}
}
\nu_{{\mathcal{P}}}
\left(
g
\left(
a_{\fp}(\psi)^2 - 4 u_{\fp}(\psi) p 
\right)
\right) 
\ll 
q^{x - \theta} \; \frac{\theta}{x} \; (x + \deg g),
\end{equation}
\begin{equation}\label{sqsieve-3}
\frac{1
}{
(\#{\mathcal{P}})
} 
\ds\sum_{
\mathfrak{p} \in {\mathcal{P}}_{\psi}
\atop{\deg \mathfrak{p} = x}
}
\nu_{{\mathcal{P}}}
\left(
g
\left(
a_{\fp}(\psi)^2 - 4 u_{\fp}(\psi) p 
\right)
\right)^2
\ll 
q^{x - 2 \theta} \; \frac{\theta^2}{x}\
\; (x + \deg g)^2.
\end{equation}

Now let  ${\ell}_1, {\ell}_2 \in {\mathcal{P}}$ be distinct primes
such that
$(\ell_1 \ell_2, M(\psi)) = 1$, where $M(\psi) \in A^{(1)}$ was introduced in part (iv) of Theorem \ref{properties-division-fields}.
(Note that, by choosing $x$ sufficiently
large,  hence, as we shall see, by choosing $\theta(x)$ sufficiently large, we can ensure that this condition holds.)
We define
$$
T_1 = T_1(\ell_1, \ell_2) 
:= 
\#\left\{ \mathfrak{p} \in {\mathcal{P}}_{\psi}: 
\deg \mathfrak{p} = x, 
\left( \frac{
a_{\mathfrak{p}}(\psi)^2 - 4 u_{\mathfrak{p}}(\psi) p}{
\ell_1 } \right) =
 \left( \frac{ a_{\mathfrak{p}}(\psi)^2 - 4
u_{\mathfrak{p}}(\psi) p}{ \ell_2 } \right) = 1 \right\},
$$
$$
T_2 = T_2(\ell_1, \ell_2)
:= \#\left\{ \mathfrak{p} \in {\mathcal{P}}_{\psi}: 
\deg \mathfrak{p} = x, 
\left( \frac{
a_{\mathfrak{p}}(\psi)^2 - 4 u_{\mathfrak{p}}(\psi) p}{
\ell_1 } \right) = \left( \frac{ a_{\mathfrak{p}}(\psi)^2 - 4
u_{\mathfrak{p}}(\psi) p}{ \ell_2 } \right) = -1 \right\},
$$
$$
T_3 = T_3(\ell_1, \ell_2) 
:= \#\left\{ \mathfrak{p} \in {\mathcal{P}}_{\psi}: 
\deg \mathfrak{p} = x, 
\left( \frac{
a_{\mathfrak{p}}(\psi)^2 - 4 u_{\mathfrak{p}}(\psi) p}{
\ell_1 } \right) = - \left( \frac{ a_{\mathfrak{p}}(\psi)^2 - 4
u_{\mathfrak{p}}(\psi) p}{ \ell_2 } \right) = 1 \right\},
$$
$$
T_4 = T_4(\ell_1, \ell_2)
 := \#\left\{ \mathfrak{p} \in {\mathcal{P}}_{\psi}:
\deg \mathfrak{p} = x, \left( \frac{
a_{\mathfrak{p}}(\psi)^2 - 4 u_{\mathfrak{p}}(\psi) p}{
\ell_1 } \right) = - \left( \frac{ a_{\mathfrak{p}}(\psi)^2 - 4
u_{\mathfrak{p}}(\psi) p}{ \ell_2 } \right) = -1 \right\},
$$	
and
$$
\hat{C}_1 
=
\hat{C}_1(\ell_1, \ell_2)
:=
\left\{
(\hat{g}_1, \hat{g}_2) \in
 \PGL_2(A/\ell_1 \ell_2 A): \left( \frac{
(\tr g_1)^2 - 4 \det{g_1}} {\ell_1} \right) = \left( \frac{(\tr g_2)^2 - 4 \det{g_2}}{\ell_2} \right) = 1 \right\},
$$
$$
\hat{C}_2 
=
\hat{C}_2(\ell_1, \ell_2)
:=
\left\{
(\hat{g}_1, \hat{g}_2) \in
\PGL_2(A/\ell_1 \ell_2 A): \left( \frac{
(\tr g_1)^2 - 4 \det{g_1}} {\ell_1} \right) = \left( \frac{(\tr g_2)^2 - 4 \det{g_2}}{\ell_2} \right)
= -1 \right\},
$$
$$
\hat{C}_3 
=
\hat{C}_3(\ell_1, \ell_2)
:=
\left\{
(\hat{g}_1, \hat{g}_2) \in
\PGL_2(A/\ell_1 \ell_2 A):
\left( \frac{ (\tr g_1)^2 - 4 \det{g_1}} {\ell_1} \right) = - \left( \frac{(\tr g_2)^2 - 4 \det{g_2}}{\ell_2} \right) =1
\right\},
$$
$$
\hat{C}_4 
=
\hat{C}_4(\ell_1, \ell_2)
:=
\left\{
(\hat{g}_1, \hat{g}_2) \in
\PGL_2(A/\ell_1 \ell_2 A):
\left( \frac{ (\tr g_1)^2 - 4 \det{g_1}} {\ell_1} \right) = - \left( \frac{(\tr g_2)^2 - 4 \det{g_2}}{\ell_2} \right) = -1
\right\},
$$
where  $\hat{g}$ denotes the projective image of a matrix $g \in \GL_2(A/\ell_1 \ell_2 A)$.

On one hand, we have
\begin{eqnarray}\label{T}
S_{\ell_1, \ell_2} 
&:=& 
\ds\sum_{
\mathfrak{p} \in {\cal{P}}_{\psi}
\atop{\deg \mathfrak{p} = x}
}
\left( 
\frac{ 
g
\left(
a_{\fp}(\psi)^2 - 4 u_{\fp}(\psi) p 
\right)
}{
{\ell}_1 }
\right)
\left( 
\frac{ 
g
\left(
a_{\fp}(\psi)^2 - 4 u_{\fp}(\psi) p 
\right)
}{
{\ell}_2 }
\right)
\nonumber
\\
&=&
 \left(\frac{g}{\ell_1}\right)
\left(\frac{g}{\ell_2}\right) \left( T_1 + T_2 - T_3 - T_4
\right).
\end{eqnarray}
On the other hand, by parts (v) and (vii) of Theorem \ref{properties-division-fields}, for each 
$1 \leq i \leq 4$ we have
\begin{equation}\label{T-Chebotarev}
T_i =
\Pi_{\hat{C}_i}
(x, J_{\ell_1 \ell_2}/F)
=
\frac{ \#\hat{C}_i }{ \#\PGL_2(A/\ell_1 \ell_2 A) }
\cdot
\frac{q^x}{x}
+
\O_{\psi}
\left(
(\#\hat{C}_i)^{\frac{1}{2}}
q^{\frac{x}{2}}
\deg(\ell_1 \ell_2)
\right).
\end{equation}

Easy counting arguments imply that, for any prime $\ell \in A$,
$$
\#\PGL_2(A/\ell A)
=
|\ell|_{\infty}  (|\ell|_{\infty}^2 - 1),
$$
$$ \#\left\{ \hat{g} \in \PGL_2(A/\ell A): 
\displaystyle \left(\frac{(\tr
g)^2 - 4 \det{g}}{\ell}\right) = 1 \right\} 
= 
\displaystyle\frac{|\ell|_{\infty}^3}{2} +
\O\left(|\ell|_{\infty}^2\right),
$$
$$ 
\#\left\{ \hat{g} \in \PGL_2(A/\ell A): \displaystyle\left(\frac{(\tr
g)^2 - 4 \det{g}}{\ell}\right) = -1 \right\} = 
\displaystyle\frac{|\ell|_{\infty}^3}{2} +
\O\left(|\ell|_{\infty}^2\right).
$$
Therefore, for each $1 \leq i \leq 4$,
$$
\#\hat{C}_i = 
\left( 
\frac{|\ell_1|_{\infty}^3}{2} 
+
\O\left(|\ell_1|_{\infty}^2\right) \right) \left( \frac{|\ell_2|_{\infty}^3}{2} 
+
\O\left(|\ell_2|_{\infty}^2\right) \right) = \frac{|\ell_1|_{\infty}^3 |\ell_2|_{\infty}^3}{4} +
\O\left(|\ell_1|_{\infty}^2 |\ell_2|_{\infty}^2 (|\ell_1|_{\infty} + |\ell_2|_{\infty})\right),
$$
where the $\O$-constants are absolute.
Consequently, by (\ref{T-Chebotarev}), for each $1 \leq i \leq 4$ we have
$$
T_i =
\frac{|\ell_1|_{\infty}^2 |\ell_2|_{\infty}^2 }{
4 \left(|\ell_1|_{\infty}^2 - 1\right)\left(|\ell_2|_{\infty}^2 - 1\right)
}
\cdot
\frac{q^x}{x}
+
\O\left(\frac{|\ell_1|_{\infty} + |\ell_2|_{\infty}}{|\ell_1|_{\infty} |\ell_2|_{\infty}} \cdot \frac{q^x}{x}\right)
+ 
\O_{\psi}\left( |\ell_1|_{\infty}^{\frac{3}{2}} |\ell_2|_{\infty}^{\frac{3}{2}} \cdot q^{\frac{x}{2}}
\log_q(|\ell_1|_{\infty} + |\ell_2|_{\infty}) \right). 
$$

By plugging these estimates into (\ref{T}) and recalling that $|\ell_1|_{\infty} = |\ell_2|_{\infty} = q^{\theta}$,
we obtain
\begin{equation}\label{sqsieve-4}
S_{\ell_1, \ell_2}
 \ll_{\psi} 
\frac{q^{x - \theta}}{x} +
q^{\frac{x}{2} + 3 \theta} \theta.
\end{equation}
Then, by combining (\ref{square-sieve-applied}) with  (\ref{sqsieve-1}), (\ref{sqsieve-2}), (\ref{sqsieve-3}), and (\ref{sqsieve-4}), we obtain
$$
S_g(\psi)
\ll_{\psi} 
q^{x - \theta} \frac{\theta}{x} (x + \deg g)
+
q^{\frac{x}{2} + 3 \theta} \theta
+
q^{x - 2 \theta} \frac{\theta^2}{x} (x + \deg g)^2.
$$

We now choose
$$
\theta := \frac{x}{8}
$$
and conclude that
$$
S_g(\psi)
 \ll_{\psi} 
 q^{\frac{7 x}{8}} (x + \deg g)
 +
 q^{\frac{3 x}{4}} x (x + \deg g)^2,
$$
justifying (\ref{square-sieve-application}).


\subsubsection{Conclusion}

By putting together (\ref{B-split}), (\ref{B1}), (\ref{B2}), and by 
choosing
$$
y(x) := \frac{11 x}{24}
$$
for any arbitrary $\varepsilon >0$,  we obtain that
\begin{equation}\label{asymp-final-noncm}
{\mathcal{B}}(\psi, x)
=
\ds\sum_{m \in A^{(1)}}
\frac{\mu_A(m) c_{J_m}(x)}{[J_m : F]}
\cdot
\frac{q^x}{x}
+
\O_{\psi, F, \varepsilon} \left(q^{\frac{23 x}{24} + x \varepsilon} x^3\right).
\end{equation}

\subsubsection{Dirichlet density}

To determine the Dirichlet density of the set $\{\fp \in {\cal{P}}_{\psi}: b_{\fp}(\psi) = 1\}$, 
we make use of the asymptotic formula  (\ref{asymp-final-noncm}). In particular, for $s > 1$ (with $s \rightarrow 1$), we have:

\begin{eqnarray*}
\ds\sum_{\fp \in {\cal{P}}_{\psi} \atop{b_{\fp}(\psi) = 1}  } q^{-s \deg \fp}
&=&
\ds\sum_{x \geq 1} q^{- s x} {\cal{B}}(\psi, x)
\\
&=&
\ds\sum_{m \in A^{(1)}}
\frac{\mu_A(m)}{[J_m : F]}
\ds\sum_{x \geq 1 \atop{c_{J_m} | x}}
\frac{q^{(1-s) x} c_{J_m}}{x}
+
\O_{\psi, F, \varepsilon}\left(
\ds\sum_{x \geq 1}
q^{\left(\frac{23}{24} + \varepsilon - s\right) x}
\right)
\\
&=&
\ds\sum_{m \in A^{(1)}}
\frac{\mu_A(m)}{[J_m : F]}
\ds\sum_{j \geq 1}
\frac{q^{(1-s) j c_{J_m}}}{j}
+
\O_{\psi, F, \varepsilon}
\left(
\frac{q^{\frac{23}{24} + \varepsilon - s} }{1 - q^{\frac{23}{24} + \varepsilon - s} }
\right)
\\
&=&
-
\ds\sum_{m \in A^{(1)}}
\frac{\mu_A(m)}{[J_m : F]}
\log 
\left(
1 - q^{(1-s) c_{J_m}}
\right)
+
\O_{\psi, F, \varepsilon}
\left(
\frac{q^{\frac{23}{24} + \varepsilon - s} }{1 - q^{\frac{23}{24} + \varepsilon - s} }
\right).
\end{eqnarray*}
Upon taking the quotient with $- \log \left(1 - q^{1-s}\right)$ and the limit $s \rightarrow 1+$, we obtain 
$\ds\sum_{m \in A^{(1)}} \frac{\mu_A(m)}{[J_m : F]}$.
We include some details for the limit of the first quotient: with $c := c_{J_m}$ and upon applying l'Hopital, we obtain
$$
\lim_{s \rightarrow 1+}
\frac{
\log \left(1 - q^{(1-s) c}\right)
}{
\log \left(1 - q^{1-s}\right)
}
=
c \ds\lim_{s \rightarrow 1+}
\frac{
q^{(1-s) c} \left(1 - q^{1-s}\right)
}{
q^{1-s} \left(1 - q^{(1-s) c}\right)
}
=
c \ds\lim_{s \rightarrow 1+}
\frac{
q^{(c-1)(1-s)}
}{
1 + q^{2 (1-s)} + q^{3 (1-s)} + \ldots + q^{(c-1) (1-s)}
}
=
\frac{c}{c}
= 1.
$$
The limit of the second quotient is $0$.

\subsection{Proof of Part (b) of Theorem \ref{bp-one}
}

With notation (\ref{defB}), we write
\begin{eqnarray}\label{B-CM}
{\mathcal{B}}(\psi, x)
&=&
\#\left\{
\mathfrak{p} \in {\mathcal{P}}_{\psi} :
\deg \mathfrak{p} = x,
\mathfrak{p} \; \text{ordinary},
E_{\psi, \mathfrak{p}}
=
A[\pi_{\fp}(\psi)]
\right\}
\nonumber
\\
&+&
\#\left\{
\mathfrak{p} \in {\mathcal{P}}_{\psi} :
\deg \mathfrak{p} = x,
\mathfrak{p} \; \text{supersingular},
E_{\psi, \mathfrak{p}}
=
A[\pi_{\fp}(\psi)]
\right\}
\nonumber
\\
&=:&
{\mathcal{B}}^{o}(\psi, x) + {\mathcal{B}}^{ss}(\psi, x).
\end{eqnarray}
We will estimate each of the two terms above separately.

\subsubsection{
Ordinary primes
}

Let $\mathfrak{p} \in {\mathcal{P}}_{\psi}$ be an ordinary prime for $\psi$. First of all, 
$$
\End_{\overline{F}}(\psi) \otimes_A F
\subseteq
\overline{E}_{\psi, \mathfrak{p}} \otimes_A F,  
$$
so using,  the assumptions that $\mathfrak{p}$ is ordinary and that 
$\End_{\overline{F}}(\psi)$ is a maximal order in $K$, we deduce that
\begin{equation}\label{endom-ordinary}
E_{\psi, \mathfrak{p}}
\simeq 
{\mathcal{O}}_{\psi, \mathfrak{p}}
\simeq
\overline{E}_{\psi, \mathfrak{p}}
\simeq
\End_{\overline{F}}(\psi).
\end{equation}
In particular, the discriminant $\Delta$ of $ \End_{\overline{F}}(\psi)$ equals $\Delta(E_{\psi, \mathfrak{p}})$ and so, 
by (\ref{ap-bp-generators}), there exists $\delta \in A$, {\it{independent of $\mathfrak{p}$}}, such that
$$
\Delta = \delta A
$$
and
$$
a_{\fp}(\psi)^2 - 4 u_{\fp}(\psi) p =
b_{\fp}(\psi)^2 \delta.
$$
Consequently,
\begin{equation}\label{bp-ordinary}
b_{\fp}(\psi) = 1
\; \;
\Leftrightarrow
\; \;
u_{\fp}(\psi) p
=
\left(\frac{a_{\fp}(\psi)}{2}\right)^2 - \frac{\delta}{4}.
\end{equation}
Recalling (\ref{RH}) and using part (i) of Lemma  \ref{lemma1}, we deduce that there are at most 
$\O(q^{\frac{x}{2}})$ possible $a_{\fp}(\psi) \in A$. 
Also, there are at most $q-1$ possible choices of $\delta$.
Thus, by (\ref{bp-ordinary}),
\begin{equation}\label{Bordinary}
{\mathcal{B}}^o(\psi, x) \ll q^{\frac{x}{2}}.
\end{equation}

\subsubsection{
Supersingular primes
}

Let $\mathfrak{p} \in {\mathcal{P}}_{\psi}$ be a supersingular prime for $\psi$. In other words,
\begin{equation}\label{ap-supersingular}
a_{\fp}(\psi) = 0
\end{equation}
(cf. \cite[Prop. 4]{Yu}). 
By using this in (\ref{ap-bp-generators}), we deduce that
$
- 4 u_{\fp}(\psi) p = b_{\fp}(\psi)^2 \delta_{\fp} (\psi),
$
which implies 
$b_{\fp}(\psi) = 1.$

Under the assumption $\End_{\overline{F}}(\psi) \otimes_A F \simeq K$, we also have that any supersingular prime 
$\fp$ for $\psi$ is either ramified or inert in 
$K$. Indeed, $K\otimes_F F_\fp$ is a subalgebra of $\overline{E}_{\psi, \mathfrak{p}} \otimes_A F_\fp$. But if $\fp$ 
is a prime of supersingular reduction, then $\overline{E}_{\psi, \mathfrak{p}} \otimes_A F_\fp$ is the  
division quaternion algebra over $F_\fp$. This implies that $K\otimes_F F_\fp$ is a field, which itself implies that $\fp$ does not split in $K$.   
Combining this with the Chebotarev Density Theorem for $K$, we deduce that
\begin{equation}\label{Bsupersingular}
{\mathcal{B}}^{ss}(\psi, x) =
 \frac{c_K(x)}{2} \cdot \frac{q^x}{x}
 +
 \O_{K}\left(q^{\frac{x}{2}}\right).
\end{equation} 

By putting together (\ref{B-CM}), (\ref{Bordinary}) and (\ref{Bsupersingular}), 
and by a similar calculation as in Section 4.2.4,
we complete the proof of part  (b) of Theorem \ref{bp-one}.

\subsection{
Remarks
}

\noindent
(i)
A natural question to ask is whether the Dirichlet density in 
part (a) of Theorem \ref{bp-one} is positive.  This question is related to a good 
understanding of the  constant $M(\psi)$ introduced in part (iv) of Theorem \ref{properties-division-fields}, and, in particular, to an 
understanding of effective versions of the Open Image Theorems for Drinfeld modules proven by R. Pink and E. R\"{u}tsche \cite{PiRu}.  
We point out that, unlike the situation for elliptic curves (cf. \cite{CoDu}, where any elliptic curve over $\Q$ with 
rational 2-torsion gives rise to a zero density of reductions with small endomorphism rings), 
there is no immediate obstruction for a Drinfeld module $\psi$ to have a positive Dirichlet density  
for $\{\fp \in {\cal{P}}_{\psi} :  
\End_{\F_{\mathfrak{p}}}(\psi \otimes \F_{\mathfrak{p}})
=
A[\pi_{\fp}(\psi)]
\}$.
In \cite{Zywina}, Zywina gives an example of a rank-$2$ Drinfeld $\F_q[T]$-module $\psi$ over $\F_q(T)$ for which 
the residual representations $\bar{\rho}_{\psi, a}$ are surjective for all $a\in A$ 
and $\overline{\F}_q\cap F(\psi[a])=\F_q$ for all $a\in A$. 
It is easy to see that for this particular $\psi$ the Dirichlet density in question is indeed non-zero. 
 
(ii) As already emphasized in Corollary \ref{structure-mod-p}, the condition $b_{\fp}(\psi) = 1$ 
implies that ${^\psi}\F_{\fp}$ is $A$-cyclic. The reductions of $\psi$ giving rise to a cyclic 
$A$-module have been studied in several works, for example,  
\cite{CoSh}, \cite{Hs}, \cite{ HsYu}, and  \cite{KuLi}. 
An outcome of part (b) of Theorem \ref{bp-one} is then that for any rank 2 Drinfeld module $\psi$ 
whose endomorphism ring is the integral closure of $A$ in a quadratic imaginary extension of $F$, 
there is a density $\geq 0.5$ of primes which give rise to reductions of $\psi$ with $A$-cyclic structures. 
This is to be contrasted with the situation for elliptic curves (see \cite{CoMu}), where  such a result is not true: 
there exist CM elliptic curves over $\Q$ (in fact, any such curve with a rational 2-torsion) 
which have no reductions with cyclic structures; moreover, for such CM elliptic curves with no rational 2-torsion 
one cannot always ensure a density  of $\geq 0.5$ of cyclic reductions.

(iii) For comparison, we recall that the Lang-Trotter Conjecture for Drinfeld modules predicts that, 
for any rank-$2$ Drinfeld module $\psi: A \to F\{\tau\}$ and any $a \in A$ (non-zero, if $\psi$
has CM), 
\begin{equation*}
{\cal{A}}(\psi, x, a)
:=
\#\{
\fp \in {\cal{P}}_{\psi}:
\deg \fp = x, 
a_{\fp}(\psi) = a
\}
\sim
C(\psi, a) \frac{q^{\frac{x}{2}}}{x}
\end{equation*}
for some constant $C(\psi, a) \geq 0$. 
Less is known about this asymptotic formula compared to what we have just proved about $b_{\fp}(\psi)$. 
Specifically, apart from lower bounds for  the case $a = 0$ (see \cite{Bro} and \cite{Da1}), 
only upper bounds for ${\cal{A}}(\psi, x, a)$ are currently known (see \cite{CoDa}, \cite{Da2}, and \cite{Zy}). 
Moreover, the particular case $a = 0$, which  is equivalent to the study of supersingular primes,  has led 
to intriguing results. Indeed, unlike for elliptic curves over $\Q$ where there are always infinitely many supersingular primes,  
there exist Drinfeld modules $\psi$ with no supersingular prime (see \cite{Po} and the references therein).


\section{CM-liftings of Drinfeld modules}\label{CMLifting}

\subsection{CM-liftings of abelian varieties} To motivate the discussion and definitions in the setting of Drinfeld modules in $\S$\ref{ssCMLD}, 
we first recall what is known about CM-liftings of abelian varieties. 

Let $B$ be an abelian variety of dimension $g$ defined over a field $K$. Following \cite[Def. 1.7]{Oo}, 
we say that $B$ has \textbf{sufficiently many complex multiplications} (or is \textbf{CM}, for short)
if $\End^0_K(B):=\End_K(B)\otimes_\Z\Q$ contains a commutative 
semi-simple algebra $L$ of dimension $2g$ over $\Q$. If $B$ is simple, then $L$ 
is necessarily a CM field, i.e., a totally imaginary quadratic extension of a totally real field.    

Let $B_0$ be an abelian variety over a field $k$ of characteristic $p$. We 
say that $B$ is a \textbf{CM-lifting} of $B_0$ if there exists a normal domain 
$R$ with fraction field $K$ of characteristic zero, a ring homomorphism $R\to k$, 
and an abelian scheme $\cB$ over $R$ such that $\cB\otimes_R k\cong B_0$ 
and $B=\cB\otimes_R K$ is CM. 

The earliest result about CM-liftings is a well-known theorem of Deuring:   

\begin{theorem}\label{Deuring}
Let $E_0$ be an elliptic curve over a finite field $k$. For any $f_0\in \End_k(E_0)$ 
generating an imaginary quadratic field $L\subset \End^0_k(E_0)$, there is an elliptic curve $E$ 
over the ring of integers $R$ of a finite extension of $\Q_p$ equipped with an endomorphism $f\in \End_K(E)$ 
such that $(E,f)$ has special fibre isomorphic to $(E_0, f_0)$. 
\end{theorem}
\begin{proof} 
See Theorem 1.7.4.6 in \cite{CCO}.  
\end{proof}

Next, as part of his proof that Tate's map from the isogeny classes of abelian varieties 
over a finite field to the Galois conjugacy classes of Weil numbers is surjective, Honda proved the 
following: 

\begin{theorem}
Given an abelian variety $B_0$ over a finite field $k$, there exists a finite 
extension $k\subset k'$ and an isogeny $B_0\otimes_k k'\to C_0$ defined over $k'$ 
such that $C_0$ has CM-lifting. 
\end{theorem}

Finally, in the recent monograph \cite{CCO} the authors show that both the isogeny and the field extension 
in the previous theorem are necessary for the existence of CM-liftings: 

\begin{theorem}\hfill
\begin{enumerate}
\item[(a)] For any $g\geq 3$, there exists an abelian variety over $\overline{\F}_p$ of dimension $g$ which 
does not admit CM-liftings. Hence the isogeny in Honda's theorem is necessary. 
\item[(b)] There exists 
an abelian variety $B_0$ over a finite field $k$ such that any $C_0$ 
isogenous to $B_0$ over $k$ does not admit a CM-lifting. 
Hence the field extension $k'/k$ in Honda's theorem is necessary. 
\end{enumerate}
\end{theorem}

\subsection{CM-liftings of Drinfeld modules}\label{ssCMLD} As at the beginning of Introduction, 
let $F$ be the function field of a smooth, projective, geometrically irreducible curve over 
$\F_q$.  Fix a place $\infty$ of $F$, and let $A$ be the subring of $F$ consisting of functions 
which are regular away from $\infty$. 

Let $R$ be a discrete valuation ring with maximal ideal $\fm$ and field of fractions $K$. 
Assume $K$ is equipped with an injective homomorphism $\gamma: A\to K$, so the $A$-characteristic of $K$ is $0$. 
A \textbf{Drinfeld $A$-module over $R$} of rank $r$ is an embedding  
$\psi: A\to R\{\tau\}$ which is a Drinfeld module over $K$ of rank $r$, as defined in Introduction, and such 
that the composite homomorphism $\overline{\psi}: A\to R\{\tau\}\to (R/\fm)\{\tau\}$
is a Drinfeld module over $R/\fm$, again of rank $r$; cf. Definition 7.1 in \cite{Ha}.  
We say that $\psi$ has \textbf{CM} if $\End_K(\psi)\otimes_A F$ is a field extension $L$ of $F$ of degree $r$. 
(Note that $L$ is imaginary.) 

Let $k$ be a finite field with $A$-characteristic $\fp$. 
Let $\psi_0$ be a Drinfeld $A$-module over $k$.  We say that $\psi_0$ has a \textbf{CM-lifting}  
if there exists a discrete valuation ring $R$ with residue field $k$, and 
a CM Drinfeld module $\psi$ over $R$ such that $\overline{\psi}$ is isomorphic to $\psi_0$ over $k$.  

Let $q^n$ be the cardinality of $k$. 
Let $\psi_0$ be a rank-$r$ Drinfeld $A$-module over $k$. Denote $E=\End_k(\psi_0)$ 
and $D=E\otimes_A F$. It is clear that $\pi:=\tau^n\in E$. Let $\widetilde{F}:=F(\pi)\subseteq D$. 
The following is known about $D$ and $\widetilde{F}$ 
(see Theorem 1 in \cite{Yu}):
\begin{itemize}
\item The degree of $\widetilde{F}$ over $F$ divides $r$. Let $t:=r/[\widetilde{F}:F]$. 
\item There is a unique place $\fP$ of $\widetilde{F}$ which is a zero of $\pi$ and there is a unique 
place $\infty_{\widetilde{F}}$ of $\widetilde{F}$ which is a pole of $\pi$. Furthermore, $\fP$ lies over $\fp$, 
and $\infty_{\widetilde{F}}$ is the unique place lying over $\infty$. 
\item $D$ is a central division algebra over $\widetilde{F}$ of dimension $t^2$ with invariants 
$$
\mathrm{inv}_v(D) = \begin{cases} 1/t &\text{if }v=\fP\\
-1/t &\text{if }v=\infty_{\widetilde{F}}\\ 
0 &\text{otherwise.}
\end{cases}
$$
\end{itemize}

By Theorem 7.15 in \cite{Re}, 
the maximal subfields of $D$ are those which have degree $r$ over $F$, and any such 
field contains $\widetilde{F}$. Let $L$ be a maximal subfield of $D$. Denote by 
$A_L$ be the integral closure of $A$ in $L$ and put $\cA=E\cap L$. 
We say that $L$ is \textbf{good} for $\psi_0$ if 
the conductor $\fc$ of $\cA$ as an $A$-order in $A_L$ is coprime to $\fp$. 

\begin{theorem}\label{thmCM} If $L$ is good for $\psi_0$,  
then the Drinfeld module $\psi_0$ has a CM-lifting $\psi$ such that $\End_K(\psi)\otimes_A F=L$. 
\end{theorem}
\begin{proof} 
We can consider $\psi_0$ as an elliptic $\cA$-module of rank $1$ defined over $k$: 
$$
\psi_0': \cA\to k\{\tau\}. 
$$
The restriction of $\psi_0'$ to $A$ is the original module $\psi_0$. 
By \cite[Prop. 4.7.19]{Go} or \cite[Prop. 3.2]{Ha}, there is a Drinfeld $A_L$-module $\phi_0'$ 
of rank $1$ over $k$, whose restriction to $\cA$ is isogenous to $\psi_0'$ over $k$. Restricting 
$\phi_0'$ to $A$ we get a Drinfeld $A$-module $\phi_0$ of rank $r$. The fact that 
$\phi_0'$ and $\psi_0'$ are isogenous, implies that there is an isogeny $i:\phi_0\to \psi_0$ over $k$. 
Moreover, since by assumption $\fc$ is coprime to $\fp$, we can choose $i$ 
so that the group-scheme $\ker(i)$ has trivial intersection with $\phi_0[\fp]$; cf. the proof of Proposition 4.7.19 in \cite{Go}.   
Now the deformation theory of Drinfeld modules 
implies that $\phi_0'$ lifts to a rank-1 Drinfeld $A_L$-module $\phi'$ over a discrete valuation ring $R$ 
whose field of fractions has zero $A$-characteristic; see \cite[$\S$3.1]{Le}. 
Restricting $\phi'$ to $A$ we get a rank-$r$ Drinfeld $A$-module $\phi$ over $K$ with CM by $L$, 
whose reduction is $\phi_0$. Since $\ker(i)$ is \'etale, Corollary 2.3 on page 42 in 
\cite{Le} implies that the kernel of $i$ lifts to an $\cA$-invariant 
submodule $H\subset {^\phi}K^{\sep}$ which is also invariant under $\Gal(K^{\sep}/K)$. 
(Note that $H$ is not necessarily $A_L$-invariant.) By \cite[Prop. 4.7.11]{Go}, 
there is an isogeny $\phi\to \psi$ defined over $K$ whose kernel is $H$. It is easy to see that $\cA\subset \End_K(\psi)$, 
and the reduction of $\psi$ is $\psi_0$, so $\psi$ is the desired 
CM-lifting of $\psi_0$. 
\end{proof}

\begin{corollary}
Any Drinfeld module $\psi_0$ is isogenous over $k$ to some 
Drinfeld module $\phi_0$ having a CM-lifting. 
\end{corollary}
\begin{proof}
This is clear from the proof of Theorem \ref{thmCM}. 
\end{proof}

\begin{proposition} In the following cases any maximal subfield $L$ is good: 
\begin{enumerate}
\item $\psi_0$ is supersingular. 
\item $r=2$.
\end{enumerate} 
\end{proposition}
\begin{proof}
Note that $\fP$ does not split in the extension $L/\widetilde{F}$. 
By Corollary to Theorem 1 in \cite{Yu}, $\cA_\fP$ is a maximal $A_\fp$ order, so 
the conductor $\fc$ is coprime to $\fP$. The Drinfeld module $\psi_0$ is supersingular if 
and only if $\fP$ is the only place of $\widetilde{F}$ over $\fp$; see \cite[(2.5.8)]{Lau}. 
These two facts imply the first claim.  Now assume $r=2$. Then either $\psi_0$ 
is supersingular, or $\widetilde{F}$ is a separable quadratic extension of $F$ and $\fp=\fP\bar{\fP}$ splits in $\widetilde{F}$. 
In the second case $L=\widetilde{F}$, and if $f(x)=x^2-ax+b=0$ is the minimal polynomial of $\pi$ 
over $F$, then $a\not\in\fp$. Note that $f'(\pi)=2\pi-a=\pi-\bar{\pi}$ is divisible neither by $\fP$ 
nor $\bar{\fP}$, so $A[\pi]$ is maximal at $\fp$; the same then is true for $E=\cA$. 
\end{proof}

By the previous proposition, if $r=2$ then any $L$ is good. 
Since any $f_0\in E$, which is not in $A$, generates 
a maximal subfield, we conclude that $(\psi_0, f_0)$ has a CM-lifting, in direct analogy with 
Deuring's Theorem \ref{Deuring}. This proves Theorem \ref{thmDLL} in the introduction.

\subsection*{Acknowledgments}
The authors thank the Institute of Mathematics at the University of G\"{o}ttingen  for hosting them in the summer of 2012, 
while some of this research was carried out. They also thank their funding agencies 
(European Research Council, National Science Foundation, and Simons Foundation)  
and the University of Illinois at Chicago for providing them with research funds  
that enabled their collaboration. They thank Florian Breuer, Chris Hall and 
Lenny Taelman for useful conversations on material related to Theorem \ref{abhyankar-thm}. 
Finally, they thank the anonymous referees for carefully reading the paper and providing them with numerous useful comments. 


\end{document}